\documentclass[11pt,reqno]{amsart}
\usepackage[english]{babel}
\usepackage{amsmath,amsfonts,amssymb,amscd,amsthm,amsbsy,amsxtra}
\usepackage{bbm,bm,epsf,calc,comment}
\usepackage{hyperref}
\usepackage{color,xcolor}
\usepackage{enumitem}
\usepackage{latexsym}
\def\e{\epsilon}
\def\R{\mathbb{R}}

\def\II{{\rm I\kern-0.5exI}}
\def\III{{\rm I\kern-0.5exI\kern-0.5exI}}

\newcommand{\norm}[1]{\lVert #1 \rVert}

\newcommand{\RR}{\mathbb{R}}

\newcommand{\ZZ}{\mathbb{Z}}

\DeclareMathOperator*{\argmin}{argmin}
\DeclareMathOperator*{\esssup}{ess\, sup}
\DeclareMathOperator*{\essinf}{ess\, inf}
\DeclareMathOperator*{\argmax}{argmax}

\DeclareSymbolFont{bbold}{U}{bbold}{m}{n}
\DeclareSymbolFontAlphabet{\mathbbold}{bbold}

\newcommand{\id}{id}

\newcommand{\vp}{\varphi}

\DeclareMathOperator{\sgn}{sgn}

\numberwithin{equation}{section}

\theoremstyle{plain}
\newtheorem{theorem}{Theorem}[section]
\newtheorem{lemma}[theorem]{Lemma}
\newtheorem{prop}[theorem]{Proposition}

\theoremstyle{definition}
\newtheorem{definition}[theorem]{Definition}

\theoremstyle{remark}
\newtheorem{remark}[theorem]{Remark}

\begin{document}
\title{The $L^1$-contraction principle in optimal transport}
\date{\today}

\author[M. Jacobs]{Matt Jacobs}
\address{Department of Mathematics, UCLA, 520 Portola Plaza, Los Angeles, CA 90095, USA}
\email{majaco@math.ucla.edu}

\author[I. Kim]{Inwon Kim}
\address{Department of Mathematics, UCLA, 520 Portola Plaza, Los Angeles, CA 90095, USA}
\email{ikim@math.ucla.edu}

\author[J. Tong]{Jiajun Tong}
\address{Department of Mathematics, UCLA, 520 Portola Plaza, Los Angeles, CA 90095, USA}
\email{jiajun@math.ucla.edu}

\begin{abstract}
In this work we use the JKO scheme to approximate a general class of diffusion problems generated by Darcy's law.  Although the scheme is now classical, if the energy density is spatially inhomogeneous or irregular, many standard methods fail to apply to establish convergence in the continuum limit. To overcome these difficulties, we analyze the scheme through its dual problem and establish a novel $L^1$-contraction principle for the density variable. Notably, the contraction principle relies only on the existence of an optimal transport map and the convexity structure of the energy. As a result, the principle holds in a very general setting, and opens the door to using optimal-transport-based variational schemes to study a larger class of non-linear inhomogeneous parabolic equations.

\end{abstract}

\maketitle

\section{Introduction}
\label{sec: intro}

Darcy's law describes a fluid flowing along a pressure gradient with the assumption that inertial forces are negligible.
     In this paper, we consider Darcy's law where the pressure is generated by a spatially inhomogeneous and convex energy functional.
To be more precise, for $T>0$, we consider %the problem
\begin{equation*}
\left\{\begin{array}{ll}
\rho_t - \nabla \cdot(\rho \nabla p ) =0 \; \hbox{ in } \Omega\times (0,T], \\
p \in \partial E(\rho),
\end{array}\right.\leqno (P)
\end{equation*}
with initial data $\rho_0$, and with $p$ satisfying the homogeneous Neumann boundary condition on $ \partial\Omega\times[0,T]$.
Here $\rho=\rho(x,t)$ represents the density of a certain material flowing in a bounded smooth domain $\Omega\subset \RR^d$ and $p = p(x,t)$ is the pressure generated by $\rho$ according to a free energy $E(\rho)$.
%More precisely, $p$ lies in the subdifferential of $E(\cdot)$ at $\rho$.
We will focus on proper, convex, lower semi-continuous energies of the form
\begin{equation}\label{eq:energy}
E(\rho) =\chi_{\Omega}(\rho)+  \int_\Omega s(\rho(x),x)\,dx,
\end{equation}
where
\[
\chi_{\Omega}(\rho)=
\begin{cases}
0 & \textrm{if} \;\; \rho(x)=0 \;\; \textrm{for all} \; x\notin \Omega,\\
+\infty & \textrm{otherwise}\\
\end{cases}
\]
restricts the density to $\Omega$, and $s(z,x) :\RR\times \Omega\to\RR\cup \{+\infty\}$ denotes a spatially inhomogeneous free energy density that is convex in the first variable. Note that when $s$ is smooth $(P)$ can be written as a parabolic PDE $\rho_t - \nabla \cdot(\rho \nabla (\partial_z s(\rho,x)))  = 0$ for nonnegative densities $\rho(x,t)$. %The conditions for $s$ are stated below in \ref{assumption: convexity}-\ref{assumption: inf sup ratio of partial s*}, where we impose only $C^1$-regularity for its dual energy $s^*$

\medskip

Since the pioneering work of Otto et al.\;\cite{jko, otto_pme}, optimal transport has been used extensively to study and model a large class of dissipative PDEs in the form of Darcy's law via minimizing movement scheme, also known in this context as the \emph{JKO scheme}.  The JKO scheme is a discrete-in-time approximation of the problem, based on the interpretation of the equation as a gradient flow of the energy $E(\rho)$ with respect to the 2-Wasserstein metric.  From a modeling perspective, the energy structure of the JKO scheme is itself meaningful and natural.   Furthermore, the JKO scheme is particularly useful for numerical simulations as the variational structure automatically provides unconditional stability.

\medskip

Our goal in this paper is to study the equation ($P$) via the JKO scheme and establish the convergence of the scheme to the continuum limit.
To study convergence, an important concept in the literature has been the geodesic convexity of the energy \cite{AmbGigSav,OTAM}.  When the energy density function $s$ is spatially inhomogenous,  geodesic convexity seems unlikely to hold in general.  Even verifying this lack of geodesic convexity seems challenging: see \cite{DiFM} for a relevant discussion in one space dimension.  In the absence of the geodesic convexity, various approaches are still available to find compactness properties to obtain convergence to the continuum limit; see for instance \cite{matthes2009,bv_ot,maury_crowd,alpar_dohyun}.  However, all of these results are limited to the first-order inhomogeneity $s(z,x)=s_1(z)+zf(x)$ and many require the domain $\Omega$ to be convex.  % \textcolor{blue}{Let us also point out that spatial inhomogenities in the energy are more versatile to handle than those in the metric, for instance the former can generate drift terms whereas the latter doesn't. See further discussion below.}

\medskip

In this paper, we establish a novel $L^1$-contraction principle for the minimizing movement scheme for a broad class of spatially inhomogeneous energy densities $s$ (see assumptions \ref{assumption: convexity}-\ref{assumption: effective domain} below).  While the $L^1$-contraction property is well-known for the continuum PDEs \cite{Carrillo,otto_contraction}, our result appears to be the first such one for the discrete solutions generated by the minimizing movements scheme.    The principle relies only on the existence of an optimal transport map and the convexity structure of the energy \eqref{eq:energy}. As a result, the principle holds in a very general setting, and it opens the door to using JKO scheme to study a large class of non-linear diffusion equations (see the discussion below).    For this reason, in the first half of the paper, we will consider a more general version of the JKO scheme where we replace the 2-Wasserstein metric with a more general transport cost function.  Here, our analysis will focus on the variational structure of the scheme, and show how the interplay between primal and dual variables leads to the $L^1$-contraction principle.  In the second half of the paper, we will return to the particular case of the 2-Wasserstein metric and use the $L^1$-contraction principle to establish convergence of the scheme to the continuum solution of $(P)$.
\medskip

Let us now introduce the minimizing movement scheme.  Given a time step $\tau>0$, the classical JKO scheme constructs an approximate solution to the PDE $(P)$ by iterating
\begin{equation}\label{eq:classic_jko}
\rho^{n+1,\tau}=\argmin_{\rho} E(\rho)+\frac{1}{2\tau}W_2^2(\rho,\rho^{n,\tau}),
\end{equation}
where $W_2^2$ is the squared 2-Wasserstein distance.  We will consider a generalized version of this variational problem where the 2-Wasserstein metric is replaced by a general optimal transport cost.
Given two nonnegative measures, $\mu$ and $\nu$, of equal mass supported on $\Omega$ and a transport cost $c:\Omega\times\Omega\to[0,\infty)$, the total transport cost between $\mu$ and $\nu$ with respect to $c$ is
\begin{equation}\label{cost}
C(\mu,\nu):=\inf_{\pi \in \Pi(\mu,\nu)} \int_{\Omega\times\Omega} c(x,y)\,d\pi(x,y),
\end{equation}
where $\Pi(\mu,\nu)$ is the set of nonnegative measures on $\Omega\times\Omega$ with first marginal $\mu$ and second marginal $\nu$.
For a given density $\bar{\rho}$, we consider the \emph{primal problem} %of minimizing
\begin{equation}\label{eq:primal_problem}
\argmin_{\rho} J(\rho,\bar{\rho}),\quad \mbox{ where } J(\rho,\bar{\rho}):= E(\rho) + C(\rho, \bar{\rho}).
\end{equation}
%Here the set $X$ will be defined below in \eqref{eqn: def of X}.
In the case $c(x,y) = \frac{1}{2\tau}|x-y|^2$, this recovers \eqref{eq:classic_jko}.

The variational problem \eqref{eq:primal_problem} is convex in $\rho$, so we can introduce the equivalent \emph{dual problem} 
\begin{equation}\label{eq:dual_problem}
\argmax_{p} J^*(p,\bar{\rho}),\quad \mbox{ where } J^*(p,\bar{\rho})=\int_{\Omega} \bar{\rho}(x) p^c(x)\, dx -E^*(p),
\end{equation}
and where the maximizer corresponds to the pressure variable in $(P)$.
Here
$$
E^*(p) :=\int_{\Omega}s^*(p(x),x)\,dx, \qquad s^*(p,x):= \sup_{z\in\mathbb{R}}\big(pz - s(z,x)\big),
$$
%the set $X^*$ will be defined in \eqref{eqn: def of X*},
%So $s^*(\cdot,x)$ is simply the convex dual of $s(\cdot,x)$.
and
\begin{equation*}
p^{c}(x):=\inf_{y\in\Omega} p(y)+c(x,y)
\end{equation*}
is the {\it $c$-transform} of $p$, which plays an essential role in the theory of optimal transport.    Let us note that the only conditions that we require for $c$ are those needed to guarantee the existence of an optimal transport map and to make $c$ behave similarly to a distance function (see \ref{assumption: symmetry of c}-\ref{assumption: twist condition} in section \ref{sec: assumptions and main results}).

\medskip

Much of our subsequent analysis  in both parts of the paper will focus on the dual problem \eqref{eq:dual_problem}.
 The advantage of the dual problem over the primal problem is that variations of the $c$-transform are easier to study than variations of the transport cost $C(\rho, \bar{\rho})$.  In addition, the optimal pressure variable has better regularity properties than the optimal density variable.   To recover information about the optimal density, we shall exploit the fact that the primal and dual variables are very closely linked.   Any optimal density
$\rho^*\in \argmin_{\rho} J(\rho,\bar{\rho})$ and any optimal pressure $p^*\in \argmax_p J^*(p,\bar{\rho})$ are linked through the duality relation $p^*\in\partial E(\rho^*)$.  Furthermore, whenever the optimal map $T$ between $\bar{\rho}$ and $\rho^*$ exists, it must solve the equation $\nabla p^*(T(x))+\nabla_y c(x,T(x))=0$ (see Proposition \ref{prop:primal_dual}).   This interplay between the primal and dual problems will be essential in our derivation of the $L^1$-contraction principle.

From the perspective of the continuum PDE $(P)$, the focus on the dual problem is also natural. Again through the equivalent duality relations $p(x)\in \partial_z s(\rho(x),x)$ and $\rho(x)\in \partial_p s^*(p(x),x)$,  the system $(P)$ can be rewritten as a nonlinear diffusion problem in terms of the pressure:
\begin{equation}\label{pressure_pde}
(a(p,x))_t -\nabla\cdot(\nabla s^*(p,x)-\vec{b}(p,x)) = 0,
\end{equation}
where $a(p,x) = \rho(x,t)= \partial_p s^*(p,x)$ and $\vec{b}(p,x)= \partial_x s^*(p,x)$. Compared to the density version of the equation, \eqref{pressure_pde} allows lower regularity for the energy density $s$.

In section \ref{sec: continumm limit}, we will show that our scheme converges to a weak solution of \eqref{pressure_pde} in the sense of \cite{Carrillo}. We further characterize the transport velocity $-\nabla p$ for $(P)$ in the density support, see Theorem~\ref{thm:pde_existence02} below. While uniqueness results hold for the spatially homogeneous case \cite{Carrillo} and for certain inhomogeneous cases (see section \ref{sec: uniqueness}), the complete uniqueness result for our notion of weak solutions remains open in general setting.

\medskip

Before stating the main results, let us emphasize that the enlarged class of costs that we consider are not artificial.   Indeed, by allowing more general costs one obtains interesting generalizations of Darcy's law.   Consider costs given by the Lagrangian action
\[
c_{\tau}(x,y)=\inf_{\gamma\in \Gamma_{\tau}(x,y)} \int_0^{\tau} L(\gamma'(t), \gamma(t))\, dt,
\]
where $\Gamma_{\tau}(x,y)=\{\gamma\in C^1([0,\tau]\to \Omega): \gamma(0)=x, \gamma(\tau)=y\}$ is the set of paths from $x$ to $y$ and the Lagrangian $L:\RR^d\times\RR^d\to[0,\infty)$ is $C^1$ and convex in the first variable. Let $H$ be the Hamiltonian corresponding to $L$, i.e.,
\[
H(q,x)=\sup_{v\in \RR^d} v\cdot q-L(v,x).
\]
Then the minimizing movement scheme with the cost $c_{\tau}$ formally approximates the PDE
\[
\left\{\begin{array}{ll}
\rho_t(x) - \nabla \cdot\big(\rho(x) \nabla_q H\big(\nabla p(x),x\big) \big) =0 \; \hbox{ in } \Omega\times (0,T]; \\
p(x) \in \partial s(\rho(x),x),
\end{array}\right.
\]
Note that in the special case $H(q,x)=\frac{1}{2}|A(x)q|^2$, where $A:\Omega\to\RR^{d\times d}$ is some non-degenerate matrix field, one obtains the anisotropic version of Darcy's law $v=-A(x)^\intercal A(x)\nabla p$.
In this special case and indeed in any case where $ \nabla_q H(\nabla p(x),x)$ is linear with respect to $\nabla p$, the techniques of this paper could be applied to obtain the convergence of the scheme to the continuum limit.  However,  when  $\nabla_q H(\nabla p(x),x)$ is non-linear with respect to $\nabla p$, passing to the limit becomes much more difficult and would require additional new ideas.  Regardless, for simplicity, when discussing continuum PDEs we shall only discuss the equation $(P)$.

\subsection{Assumptions and main results}
\label{sec: assumptions and main results}
Let us now introduce precise assumptions on the energy and the cost functions. For the $L^1$-contraction principle, we only need the following minimal assumptions on the energy functional, which will be assumed throughout the paper.

\begin{enumerate}[label=(s\arabic*)]
\setcounter{enumi}{0}
\item \label{assumption: convexity} For all $x\in\Omega$, $s(\cdot,x)$ is a proper, lower semi-continuous, convex function.
\item \label{assumption: effective domain} $s(z,x) \equiv+\infty$ if $z<0$, $s(0,x) = 0$, %\textcolor{red}{(It seems that it is not as general as $s(0,x)<+\infty$ a.e.\;since this will interact with \ref{assumption: regularity of s*}?)}
\[
-\infty<  \inf_{(z,x)\in\RR\times\Omega} s(z,x), \quad \textrm{and}\quad \lim_{z\to+\infty} \inf_{x\in\Omega} \dfrac{s(z,x)}{z}=+\infty.
 \]
\end{enumerate}

For the cost function in \eqref{cost}, we always assume that it satisfies
\begin{enumerate}[label=(c\arabic*)]
\setcounter{enumi}{0}
\item \label{assumption: symmetry of c} symmetry: $c(x,y)=c(y,x)$, and $c(x,x)=0$.
\item \label{assumption: regularity of c} continuity: $c\in C^1_{loc}(\RR^d\times\RR^d\to [0,\infty))$.
\item \label{assumption: twist condition} the twist condition: for all $x_0\in \RR^d$ the map $y\mapsto \nabla_x c(x_0,y)$ is injective over $\RR^d$.
\end{enumerate}
These assumptions guarantee that there exists an optimal transport map for the cost $c$ between any two absolutely continuous measures with the same mass \cite{OTAM}.

\medskip

For the convergence analysis in the second part of our paper, we will require further regularity properties for the dual energy $s^*$ as follows:
\begin{enumerate}[label=(s\arabic*)]
\setcounter{enumi}{2}

\item \label{assumption: regularity of s*} $s^*$ is differentiable, with its derivatives $\partial_p s^*(p,x)$ and $\nabla_x s^*(p,x)$ continuous with respect to $p$ for all $x\in \Omega$. Moreover for any finite $b$ we have
\[
\int_{\Omega} \norm{\nabla_x s^*(\cdot,x)}^2_{C\big((-\infty,b)\big)}\, dx<+\infty.
\]

\item \label{assumption: asymptotics of s*} $\lim_{\alpha\to -\infty} \esssup_{x\in\Omega}\, \partial_p s^*(\alpha,x)=0$, and $\lim_{\alpha\to+\infty} \essinf_{x\in\Omega} \partial_p s^*(\alpha,x)>0$.

\item \label{assumption: inf sup ratio of partial s*} For any $\alpha\in \RR$ there exists $M_{\alpha}\in (0,\infty)$ such that
    \[
    \esssup_{x\in \Omega} \partial_p s^*(\alpha, x)\leq M_{\alpha}\essinf_{x\in\Omega} \partial_p s^*(\alpha,x).
    \]
\end{enumerate}

\begin{remark}
The assumption \ref{assumption: regularity of s*} ensures sufficient regularity for the existence of weak solutions to \eqref{pressure_pde}. Note that \ref{assumption: regularity of s*} is equivalent to the strict convexity of $z\mapsto s(z,x)$ for $z\in \partial s^*(\RR, x)$.
\ref{assumption: asymptotics of s*} guarantees the existence of a stationary solution to the primal problem with a given mass, while \ref{assumption: inf sup ratio of partial s*} ensures that the flow cannot spontaneously form a vacuum when the density is everywhere bounded away from zero. Together, assumptions \ref{assumption: asymptotics of s*} and \ref{assumption: inf sup ratio of partial s*} allow us to construct barriers that give uniform lower bounds on solutions starting from certain strictly positive initial data.   Crucially, these barriers will allow us to approximate any solution with a non-degenerate and regular solution.
\end{remark}
\begin{remark}
Let us note that a simple class of energy densities satisfying \ref{assumption: convexity}-\ref{assumption: inf sup ratio of partial s*} are given by the multiplicative structure $s(z,x) = f(x) g(z)$, where $f:\bar{\Omega}\to \RR$ is smooth and strictly positive and $g$ is a convex and superlinear function such that $\partial g(0)\subset [-\infty,0]$ and $\partial g(z)\subset (0,\infty)$ if $z>0$.  This includes many familiar choices, for example,
$$
g(z)= z\ln z -z,  \quad  g(z) =\frac{1}{m-1} z^m\,\,  (m>1), \quad\mbox{ or } \quad g(z) = \frac{1}{m-1} z^m + z^2 \,\, (m<1).
$$
One can then create a larger class of energy densities by taking sums $s(z,x)=\sum_{i=1}^k f_i(x)g_i(z)$ or infimal convolutions $ s(z,x)=\inf_{\sum_{i=1}^k z_i=z} \sum_{i=1}^k f_i(x)g_i(z_i)$ of copies of the multiplicative structure described above.

\end{remark}

To state the main results, first let us only assume $s$ satisfies \ref{assumption: convexity}-\ref{assumption: effective domain} and $c$ satisfies \ref{assumption: symmetry of c}-\ref{assumption: twist condition}.
Define
\begin{equation}\label{eqn: def of X}
X:=\{\rho\in L^1(\Omega): E(\rho)<\infty\}
\end{equation}
and
\begin{equation}\label{eqn: def of X*}
X^*:=\{p:\Omega\to [-\infty,+\infty]:\, p\mathrm{\;is\;measurable\;and\;}E^*(p)<\infty\}.
\end{equation}
For the primal problem \eqref{eq:primal_problem}, we assume the initial density to satisfy
\begin{equation}\label{eq:good_data0}
0<\int_{\Omega}\bar{\rho}(x)\, dx<\lim_{b\to\infty} \int_{\Omega} \sup\partial s^*(b,x) dx.
\end{equation}
%so that the corresponding pressure variable will be finite.
Here $\partial s^*(b,x)$ denotes the subdifferential of $s^*(\cdot,x)$ with respect to the first variable at $b\in\mathbb{R}$.
It is noteworthy that \eqref{eq:good_data0} can be equivalently written as
\[
\lim_{b\to-\infty} \int_{\Omega} \sup\partial s^*(b,x) dx<\int_{\Omega}\bar{\rho}(x)\, dx<\lim_{b\to\infty} \int_{\Omega} \inf\partial s^*(b,x)\, dx.
\]
It will be shown in Lemma \ref{lem:s_star_increasing} that $\lim_{b\to-\infty} \int_{\Omega} \sup\partial s^*(b,x) dx=0$.   As for switching the sup to an inf in the upper bound,  since $\partial s^*(\cdot,x)$ is increasing, it is true that for all $x\in \Omega$, $b\in\mathbb{R}$ and $\varepsilon>0$,
\[
\inf\partial s^*(b,x)\leq \sup\partial s^*(b,x)\leq \inf\partial s^*(b+\varepsilon,x).
\]
Hence, the limits must be the same.
 \eqref{eq:good_data0} guarantees that the corresponding maximizing pressure in the dual problem \eqref{eq:dual_problem} will be finite (see the proof of Proposition \ref{prop:primal_dual}).

\begin{theorem}[Discrete $L^1$-contraction property, Theorem \ref{thm:l1_contraction}]
\label{thm:l1_contraction_intro}
%Suppose $s$ satisfies \ref{assumption: convexity}-\ref{assumption: effective domain}.
Let $\rho_0, \rho_1\in X$ satisfy \eqref{eq:good_data0}, and %let
\[
\rho_i^*=\argmin_{\rho\in X} J(\rho,\rho_i),
\]
where $J(\rho,\rho_i)$ is defined in \eqref{eq:primal_problem}.
Then we have
\[
\norm{(\rho_1^*-\rho_0^*)_+}_{L^1(\Omega)}\leq \norm{(\rho_1-\rho_0)_+}_{L^1(\Omega)}.
\]
\end{theorem}
\begin{remark}
In the special case where the energy is translation invariant (i.e.\;$\Omega=\RR^d$ and $s(z,x)=s(z)$),  our contraction result should result in BV estimates akin to that of \cite{bv_ot}.   The translation invariance allows one to apply the contraction principle to the difference between a density and its translated version.  Hence, for any $y\in \RR^d$ one has the inequality
 \[
 \int_{\RR^d} |\rho^*(x+y)-\rho^*(x)|\leq \int_{\RR^d} |\rho(x+y)-\rho(x)|
 \]
 which can be readily converted into a BV norm inequality.  Note that this will also hold for a general domain $\Omega$ and a spatially homogeneous energy, as long as one knows that the density stays compactly supported away from the boundary.
\end{remark}
\medskip

Based on Theorem~\ref{thm:l1_contraction_intro}, we obtain the discrete comparison principle for both density and pressure variable. To our best knowledge, the only previous comparison result  is \cite{AKY}, which addresses the particular case of $s(z)=\frac{1}{m-1}z^m$ with
$m>1$ and the quadratic cost. In addition to the inherent interest of the comparison principle, it will also prove to be useful in our approximation argument (see the discussion below).

\begin{theorem}[Discrete comparison principle, Lemma \ref{lem:big_max_unique}, Lemma \ref{lem:pressure_special}, and Theorem \ref{thm: direct comparison of rho and p}]
\label{thm: direct comparison of rho and p intro}
Under the assumptions of Theorem~\ref{thm:l1_contraction_intro}, suppose $\rho_0\leq \rho_1$ a.e.\;in $\Omega$.
Then
\begin{enumerate}
\item $\rho_0^*\leq \rho_1^*$ a.e.\;in $\Omega$.

\item
With $i = 0,1$, there exists a largest and a smallest $c$-concave maximizing pressure in
\[
\argmax_{p\in X^*,\,p^{c\bar{c}}=p} J^*(p,\rho_i),
\]
denoted by $p_i^+$ and $p_i^-$ respectively, in the sense that $p_i^-(x)\leq \tilde{p}_i(x)\leq p_i^+(x)$ in $\Omega$ for any
\[
\tilde{p}_i\in \argmax_{p\in X^*,\,p^{c\bar{c}}=p} J^*(p,\rho_i).
\]
Here $J^*(p,\rho_i)$ is defined as in \eqref{eq:dual_problem}, and
\begin{equation*}
q^{\bar{c}}(x):=\sup_{y\in \Omega} q(y)-c(x,y)
\end{equation*}
is the $\bar{c}$-transform.
$p$ is called $c$-concave if and only if $p^{c\bar{c}} = p$.

For such $p_i^\pm$, we have $p_0^+\leq p_1^+$ and $p_0^-\leq p_1^-$ in $\Omega$.
\end{enumerate}
\end{theorem}

The remainder of our main results concern the second part of our paper where we show that the discrete solutions obtain from the minimizing movement scheme converge to solutions of the problem $(P)$.
The contraction principle will play a crucial role in our arguments.

$s$ is now assumed to satisfy all the assumptions \ref{assumption: convexity}-\ref{assumption: inf sup ratio of partial s*}, and the cost is specialized to the quadratic cost $c(x,y):= \frac{1}{2\tau}|x-y|^2$, where $\tau$ plays the role of a time step. We shall suppose that the initial data $\rho_0$ satisfies a slightly stronger version of \eqref{eq:good_data0}, i.e.,
\begin{equation}\label{good_data_1}
\int_\Omega\rho_0\,dx>0,\quad \mbox{and}\quad\rho_0 \leq \partial_p s^*(M,\cdot)\mbox{ a.e.\;in } \Omega
\end{equation}
for some $M<+\infty$. In fact, this will imply that $\rho_0$ admits a pressure variable bounded from above, and that $\rho_0$ satisfies \eqref{eq:good_data0}.

Setting $\rho^{0,\tau} = \rho_0$, we construct approximate solutions to ($P$) by iterating:
\begin{equation}\label{eq:mms}
\rho^{n+1,\tau}:=\argmin_{\rho\in X}  J(\rho,\rho^{n,\tau}),
\end{equation}
and
\begin{equation}\label{eq:mms_dual}
p^{n+1,\tau}\in \argmax_{p\in X^*,\,p^{c\bar{c}} = p}  J^*(p,\rho^{n,\tau})
\end{equation}
where $p^{n+1,\tau}$ is chosen as the smallest $c$-concave maximizer given in Theorem \ref{thm: direct comparison of rho and p intro}.  Note that now the primal problem has the more familiar form \eqref{eq:classic_jko}.

Define $\rho^{\tau}:\Omega\times[0,\infty)\to [0,\infty)$ and  $p^{\tau}:\Omega\times[0,\infty)\to \RR$ as the piecewise constant interpolations on $[0,\infty)$ of the discrete solutions
\begin{equation}\label{interpolation}
\rho^{\tau}(x,t):=\rho^{n+1,\tau}(x),\quad p^{\tau}(x,t):=p^{n+1,\tau}(x) \quad\hbox{ for } n\tau \leq t < (n+1)\tau.
\end{equation}
Our goal is to show that as $\tau\to 0$, $(\rho^\tau,p^\tau)$ converges up to a subsequence to a weak solution of the problem $(P)$.

\medskip

The starting point of the analysis is to use the $L^1$-contraction to establish spatial equicontinuity of $\rho^\tau$ (see Proposition \ref{prop:spatial_equicontinuity}) which then yields the $L^1$-convergence of $\rho^\tau$ in space-time.
As for the compactness of $p^\tau$, the available energy dissipation inequality (see Lemma \ref{lem: edi}) only bounds the integral of the combined quantity $\rho^{\tau}|\nabla p^{\tau}|^2$, and thus additional ideas are needed to discuss the convergence of $p^{\tau}$ in view of potential degeneracy of $\rho^{\tau}$ (i.e.\;$\rho^\tau = 0$ in some region of the space-time domain).

First we show that, for strictly positive initial density $\rho_0$, $(\rho^{\tau}, p^{\tau})$ converges to the standard weak solution of $(P)$ where $\nabla p$ is obtained as an $L^2$-function in space-time. A crucial ingredient is Lemma~\ref{lem:lower_bound}, by which stationary solutions can serve as barriers to provide uniform bounds on $\rho^\tau$ and $p^\tau$. For a given $T>0$, we denote $\Omega_T:= \Omega\times [0,T]$. The convergences as $\tau\to 0$ are subsequential.

\begin{theorem}[Theorem \ref{thm:pde_existence1}]\label{thm:pde_existence01}
Suppose $\rho_0\in X$ satisfies \eqref{good_data_1}.
In addition, suppose $\rho_0(x)\geq \partial_p s^*(m,x)$ for some constant $m\in \RR$ such that $\partial_p s^*(m,x)$ is not identically zero.
Then for any $T>0$, there exist $\rho\in L^{\infty}\big(\Omega_T\big)$ and $p\in L^2\big([0,T];H^1(\Omega)\big)\cap L^{\infty}(\Omega_T) $ such that, up to a subsequence,
$\rho^{\tau} \to \rho \hbox{ in }L^1\big(\Omega_T\big)$ and $p^{\tau}\rightharpoonup p$  in $L^2([0,T]; H^1(\Omega))$.
Moreover, $(\rho,p)$ is a weak solution of $(P)$ in the sense that
$p(x,t)\in \partial_p s(\rho(x,t),x)$ a.e.\;in $\Omega_T$
and
$$
\int_0^{t_0}\int_{\Omega} \rho (x,t)\partial_t \phi(x,t)-\rho(x,t)\nabla p(x,t)\cdot \nabla \phi(x,t)\, dx\, dt=\int_{\Omega} \rho(x,t_0)\phi(x,t_0)-\rho_0(x)\phi(0,x)\,dx
$$
for any $\phi \in C^{\infty}(\Omega_T)$ and for a.e. $t_0\in [0,T]$.
\end{theorem}

For general bounded initial data $\rho_0$ satisfying \eqref{good_data_1}, we obtain a weaker notion of continuum solutions of $(P)$ in the context of \eqref{pressure_pde}, similar to that of \cite{Carrillo}. This notion still admits uniqueness results on the $\rho$-variable for a wide class of $s^*$ (see section \ref{sec: uniqueness}).
Instead of directly taking the limit $\tau\to 0$ in the discrete solutions $(\rho^\tau,p^\tau)$,
our solution is obtained by an approximation argument using solutions starting from strictly positive initial data.
In this process, the $L^1$-contraction and the comparison principle play crucial roles.
See the details in section \ref{sec: continumm limit}.

\begin{theorem}[Theorem \ref{thm:pde_existence2}]\label{thm:pde_existence02}
Suppose $\rho_0\in X$ satisfies \eqref{good_data_1}. Then for any $T>0$  there exists $\rho\in L^{\infty}\big(\Omega_T\big)$ and a measurable $p$ with $p_+:=\max\{p,0\} \in L^{\infty}\big( \Omega_T\big)$, such that
$\rho^{\tau} \to \rho$ in $L^1(\Omega_T)$ along a subsequence, and
$\rho(x,t)=\partial_p s^*(p(x,t),x) \hbox{ a.e.\;in }\Omega_T$.
Moreover $(\rho,p)$ is a weak solution of \eqref{pressure_pde} in the sense that  for any $\phi\in C^{\infty}(\Omega_T)$ and for a.e. $t_0\in [0,T]$
\begin{equation*}%\label{weak_eq}
\int_0^{t_0}\int_{\Omega} \rho(x,t)\partial_t \phi(x,t)-m(x,t)\cdot \nabla \phi(x,t)\, dx\, dt=\int_{\Omega} \rho(x,t_0)\phi(x,t_0)-\rho_0(x)\phi(x,0)\,dx,
\end{equation*}
where $m(x,t) := \nabla [s^*(p(x,t),x)] - \partial_x s^*(p(x,t),x)\in L^2(\Omega_T)$.  Lastly, $\frac{m}{\rho} = \nabla p$ in the support of $\rho$, in the sense that
 $$
\int_{\Omega_T} \frac{m}{\rho} \cdot f = -\int_{\Omega_T} p \nabla\cdot f
$$
for any $d$-dimensional vector field $f\in L^2([0,T]; H^1(\Omega))$ with $\|\rho^{-1}f \|_{L^2(\Omega_T)}+\|\rho^{-1}\nabla\cdot f\|_{L^1(\Omega_T)}<+\infty$ and with zero normal component along $\partial\Omega\times [0,T]$.
\end{theorem}

Finally we discuss uniqueness of weak solutions, extending the results of \cite{Carrillo} and \cite{vazquez} to the spatially inhomogeneous cases. To ensure uniqueness, we impose the structural condition \eqref{assumption_s}, which guarantees sufficient regularity of the drift term.    In particular, our weak solution formulation does not require the additional entropy condition considered in \cite{Carrillo} to obtain uniqueness.

\begin{theorem}[Theorem \ref{thm:uniqueness}]\label{thm:uniqueness}
Let $\rho_0$ be as given in Theorem~\ref{thm:pde_existence02}, and suppose $s$ either is of the form $s^*(p,x) = f(x)w(p)$ or satisfies \eqref{assumption_s}. Then the density variable $\rho$ in the weak solution of (P) constructed in Theorem \ref{thm:pde_existence02} is unique, and the entire sequence $\rho^{\tau}$ converges to $\rho$.
The $m$-variable in Theorem \ref{thm:pde_existence02} is also unique.

Moreover in this case the $L^1$-contraction carries over to the continuum solutions.
Namely, let $\rho_{0,i}$ $(i=1,2)$ be given as in Theorem~\ref{thm:pde_existence02}, and let $\rho_i$ be the corresponding density variables.
Then
$$
\|\rho_1(\cdot,t) - \rho_2(\cdot,t) \|_{L^1(\Omega)} \leq \|\rho_{0,1}- \rho_{0,2}\|_{L^1(\Omega)} \hbox{ for any } t>0.
$$
\end{theorem}

\subsection{Organization of the paper}

As mentioned above, the paper consists of two parts. Sections \ref{sec:background}-\ref{sec: L^1 contraction} establish important general principles %, such as the $L^1$-contraction property and the comparison principle,
of the minimizing movement scheme with general costs, whereas sections \ref{sec: mms and equicontinuity}-\ref{sec: uniqueness} turn to the quadratic cost, with focus on establishing convergence of the discrete solutions as $\tau\to 0$.

In section \ref{sec:background}, we recall basic properties of optimal transport and convex duality theory.
In section \ref{sec:dual}, we discuss the equivalence between the primal and the dual problems, as well as the existence and uniqueness of their solutions (Proposition \ref{prop:primal_dual}).
We prove the existence of the largest $c$-concave maximizing pressure in Lemma \ref{lem:big_max_unique}, and a maximum-principle-type result for the pressure in Proposition \ref{prop:maximum_principle}.
Section \ref{sec: L^1 contraction} features the $L^1$-contraction principle (Theorem \ref{thm:l1_contraction}) and the comparison principle (Theorem \ref{thm: direct comparison of rho and p}).
The existence of the smallest $c$-concave maximizing pressure is also proven in Lemma \ref{lem:pressure_special}.

Section \ref{sec: mms and equicontinuity} establishes strong compactness of the discrete density variable $\rho^{\tau}$ in $L^1(\Omega_T)$ (see Proposition \ref{prop:strong_compactness}).
The main ingredient needed for compactness is the $L^1$-equicontinuity of $\{\rho^\tau\}_\tau$ (Proposition~\ref{prop:spatial_equicontinuity}), which we obtain by combining the $L^1$-contraction principle with an approximation argument.
Section \ref{sec: continumm limit} justifies the convergence to the continuum solutions stated in Theorems \ref{thm:pde_existence01} and \ref{thm:pde_existence02}, making use of the strong compactness of $\{\rho^\tau\}_\tau$ in section \ref{sec: mms and equicontinuity}, uniform bounds coming from the comparison principle, the energy dissipation inequality, and the dual relation between $\rho^\tau$ and $p^\tau$.
Lastly, section \ref{sec: uniqueness} yields a uniqueness result for the $\rho$- and $m$-variables of the weak solutions obtained in Theorem \ref{thm:pde_existence02}, under additionally assumptions on $s^*(p,x)$.
This generalizes the uniqueness result of V\'{a}zquez \cite{vazquez}, which considered spatially homogeneous energy densities.

\subsection{Acknowledgement} I.K.\;thanks Felix Otto for helpful discussions, in particular motivating our investigation on the $L^1$-contraction for the discrete scheme. I.K.\;also thanks Katy Craig for helpful discussions on the geodesic convexity and pointing to the reference \cite{DiFM}.  The authors are grateful to Alp\'ar M\'esz\'aros for helpful comments.     M.J.\;is supported by ONR N00014-18-1-2527 and AFOSR MURI FA9550-18-1-0502. I.K.\;is supported by NSF grant DMS-1900804 and the Simons Fellowship.

\section{Preliminary Results}\label{sec:background}

\subsection{Properties of the optimal transport}
We first list some essential properties of optimal transport. Since we primarily work with optimal transport in its dual formulation, we shall work extensively with the $c$-transform.
Recall that we always assume that $c$ satisfies \ref{assumption: symmetry of c}-\ref{assumption: twist condition}.

\begin{definition} Given a function $p:\Omega\to\RR$ the {\it $c$-transform} of $p$ is given by
\begin{equation}\label{c-transform}
p^{c}(y)=\inf_{x\in\Omega} p(x)+c(x,y).
\end{equation}
Given a function $q:\Omega\to\RR$  the {\it  conjugate $c$-transform} is given by
\begin{equation*}
q^{\bar{c}}(x):=\sup_{y\in \Omega} q(y)-c(x,y).
\end{equation*}
\end{definition}
\begin{remark}
Note that there is no universally-agreed-upon choice of sign convention for the $c$-transform.  We choose the convention that leads to the simplest notation for our variational problems.
\end{remark}

\begin{lemma}[\cite{OTAM}]\label{lem:ccc}
Given functions $p, q:\Omega\to\RR$, we have
\[
p^{c\bar{c}}\leq p, \quad q\leq q^{\bar{c}c},
\]
and
\[
p^{c\bar{c}c}=p^{c},\quad q^{\bar{c}c\bar{c}}=q^{\bar{c}}.
\]
\end{lemma}

\begin{definition}
We say that a function $p:\Omega\to \RR$ is {\it $c$-concave} if $p^{c\bar{c}}=p$, and we say a pair of functions $p, q:\Omega\to\RR$ are {\it $c$-conjugate} if $p^c=q$ and $q^{\bar{c}}=p$.
\end{definition}

The following regularity result is a well-known consequence of the $c$-transform definition.

\begin{lemma}[\cite{OTAM}]\label{lem:lip}
If $p$ is $c$-concave, then $p$ is Lipschitz and the Lipschitz constant depends only on $c$ and $\Omega$.
\end{lemma}

The following two lemmas establish the fundamental relationship between optimal transport and the $c$-transform.
\begin{lemma}[\cite{OTAM}]\label{lem:c_duality}
If $\mu$ is a nonnegative measure, then for any bounded function $p:\Omega\to \RR$,
\[
\inf_{\rho\in L^1(\Omega), \rho(\Omega)=\mu(\Omega)}\, \int_{\Omega}p(x)\rho(x)\, dx+C(\rho, \mu)=\int_{\Omega} p^c(y)\, d\mu(y).
\]
\end{lemma}

\begin{lemma}[\cite{gangbo_polar, gangbo_habilitation, gangbo_mccann}] \label{lem:c_transform_variation}
If $p:\Omega\to\RR$ is $c$-concave, $\mu$ is a nonnegative measure, and $\phi:\Omega\to\RR$ is a continuous function, then
\[
\lim_{t\to 0^+} \int_{\Omega} \frac{(p+t\phi)^c(y)-p^c(y)}{t}\,d\mu(y)=\int_{\Omega} \phi(T_{p}(y))\,d\mu(y)
\]
where $T_p:\Omega\to\Omega$ is the almost everywhere unique solution to
\begin{equation}\label{eq:forward_map}
\nabla p(T_p(y))+\nabla_x c(T_p(y),y)=0.
\end{equation}
Furthermore, $T_{p}$ is invertible for almost every $y\in\Omega$, and $T_p^{-1}$ is the almost everywhere unique solution to
\begin{equation}\label{eq:backward_map}
\nabla p(x)+\nabla_x c(x,T_{p}^{-1}(x))=0.
\end{equation}
\end{lemma}
\begin{remark}
The maps $T_p$ and $T_p^{-1}$ can additionally be characterized as the unique solutions to the optimization problems:
\begin{equation}\label{eq:map_argmin}
T_p(y)=\argmin_{x\in\Omega} p(x)+c(x,y),\quad T_p^{-1}(x)=\argmax_{y\in\Omega} p^c(y)-c(x,y).
\end{equation}
\end{remark}
\begin{remark}
If $c(x,y)=\frac{1}{2\tau}|x-y|^2$, the maps $T_p$ and $T_p^{-1}$ are given by
\begin{equation}\label{eq:quadratic_p_map}
T_p(y)=y-\tau\nabla p^c(y),\quad T_p^{-1}(x)=x+\tau\nabla p(x).
\end{equation}
\end{remark}

Now we can finally state the fundamental result that guarantees the existence and uniqueness of the optimal transport maps.
\begin{theorem}[\cite{brenier_polar,gangbo_habilitation,gangbo_mccann}] \label{thm:fund_ot}
If $\mu, \nu\in L^1(\Omega)$ are nonnegative densities with the same mass, then there exists a $c$-concave function $p^*:\Omega\to\RR$ such that
\[
p^*\in\argmax_p \int_{\Omega} p^c(y) \mu(y)\, dy-\int_{\Omega} p(x)\nu(x)\, dx,
\]
\[
C(\mu,\nu)=\int_{\Omega} (p^*)^c(y) \mu(y)\, dy-\int_{\Omega} p^*(x)\nu(x)\, dx,
\]
$T_{p^*}$ is the unique optimal map (up to a $\mu$-measure-zero set) transporting $\mu$ to $\nu$, and $T_{p^*}^{-1}$
is the unique optimal map (up to a $\nu$-measure-zero set) transporting $\nu$ to $\mu$.

Conversely, if $\tilde{p}$ is a $c$-concave function such that $T_{\tilde{p}\,\#}\mu=\nu$, then $T_{\tilde{p}}$ is the unique optimal map (up to a $\mu$-measure-zero set) transporting $\mu$ to $\nu$  and $T_{\tilde{p}}^{-1}$ is the unique optimal map (up to a $\nu$-measure-zero set) transporting $\nu$ to $\mu$.
\end{theorem}

\subsection{Properties of the convex duals}
Next we review several useful properties of the convex duals that will be used throughout the paper.

\begin{lemma}\label{lem:dual_relation}
For any proper, lower semi-continuous, convex function $h:\RR\to\RR\cup\{+\infty\}$, we have $p\in \partial h(y)$ if and only if $py=h(y)+h^*(p)$.
%where $h^*$ is the convex dual
Moreover, $p\in \partial h(y)$ if and only if $y\in \partial h^*(p)$.

\begin{proof}
First suppose $p\in \partial h(y)$.
This implies that for any $z\in\RR$, $h(z)\geq h(y)+p(z-y)$.
Hence, $py\geq h(y)+pz-h(z)$ for any $z\in\RR$.
Taking the supremum over $z\in \R$, we get $py\geq h(y)+h^*(p)$.
The opposite direction $py\leq h(y)+h^*(p)$ immediately follows from Young's inequality.

On the other hand, suppose that $py=h(y)+h^*(p)$. Then
\[
h(y)=py+\inf_z h(z)-pz\leq h(y')+p(y-y')
\]
for any $y'\in \RR$.  Thus $p\in\partial h(y)$.

The second claim immediately follows from the first one if one notices that $h^*$ is a proper, lower semi-continuous, convex function on $\RR$ with value in $\RR\cup\{+\infty\}$.
\end{proof}
\end{lemma}

\begin{lemma}\label{lem:s_star_increasing}
Suppose $h:\RR\to\RR\cup\{+\infty\}$ is a proper, lower semi-continuous, convex function, satisfying $h(x) = +\infty$ if $x<0$ and $h(0)=0$.
Then $h^*$ is nonnegative, increasing, and in fact strictly increasing on $\partial h((0,\infty))$.
Moreover, $\sup \partial h^*(p) \to 0$ as $p\to-\infty$; here the supremum is taken over the elements in the subdifferential $\partial h^*(p)$.

\begin{remark}\label{rmk: s star increasing}
By the assumption on $s(\cdot,x)$, we immediately know that for all $x\in \Omega$, $s^*(\cdot,x)$ is nonnegative, increasing, and strictly increasing on $\partial s((0,\infty),x)$.
\end{remark}
\begin{proof}
By the assumption on $h$,
\[
h^*(p) = \sup_{z\geq 0}\;pz-h(z).
\]
Hence,
\[
h^*(p)\geq -h(0)=0.
\]
If $p_1\leq p_2$, then $p_1z\leq p_2z$ for all $z\geq 0$, so $h^*(p_1)\leq h^*(p_2)$.
If $p_1\in \partial h\big(y_0)$ for some $y_0>0$ and $p_2>p_1$, then
\[
h^*(p_2)=\sup_{y\geq 0}\, p_2y-h(y)\geq p_2y_0-h(y_0)>p_1y_0-h(y_0)=h^*(p_1).
\]
We used Lemma \ref{lem:dual_relation} in the last equality.

Lastly, suppose that $\sup \partial h^*(p) \to 0$ as $p\to-\infty$ does not hold. Since $\sup\partial h^*(p)$ is non-decreasing in $p$ and non-negative, we must have that $\sup \partial h^*(p) \geq c$ for all $p\in \mathbb{R}$ with some $c>0$.
Hence, for any $p\leq p'$,
\[
h^*(p')\geq h^*(p)+(p'-p)\cdot\sup \partial h^*(p)\geq c(p'-p).
\]
Letting $p\to -\infty$ yields that $h^*(p')=+\infty$ for all $p'\in \mathbb{R}$, which leads to a contradiction.
\end{proof}
\end{lemma}

\begin{lemma}\label{lem:pressure_upper_bound}
Suppose $p$ is a measurable function on $\Omega$ such that $p\leq M$ for some $M$ finite, and $\rho(x)\in \partial s^*(p(x),x)$ a.e.\;in $\Omega$.
Then $\rho\in L^{\infty}(\Omega)$ and $\rho p\in L^{\infty}(\Omega)$ and both bounds depend only on $s$ and $M$.
\end{lemma}
\begin{proof}
%Given $M$, by  there exists $z_M>0$ such that
%$$
%\frac{s(z,x)}{z}\geq M\quad \mbox{whenever}\quad z\geq z_M.
%$$
The condition $\rho(x)\in \partial s^*(p(x),x)$ a.e.\;in $\Omega$ is equivalent to  $\rho(x)\in \argmax_{z\geq 0} zp(x)-s(z,x)$ a.e.\;in $\Omega$.
From \ref{assumption: effective domain} we have
\[
\lim_{z\to\infty}  \sup_{x\in\Omega} zp(x)-s(z,x)\leq \lim_{z\to\infty} \sup_{x\in\Omega} zM-s(z,x)=-\infty.
\]
Since $\rho(x)\in \argmax_{z\geq 0} zp(x)-s(z,x)$ for almost every $x$ and choosing $z=0$ always gives the value $0$, it follows that $\rho\in L^{\infty}(\Omega)$ with a bound that depends only on $s$ and $M$.

By Lemma \ref{lem:dual_relation}, that $\rho(x)\in \partial s^*(p(x),x)$ is equivalent to
\[
\rho(x)p(x)=s(\rho(x),x)+s^*(p(x),x).
\]
So we can derive that
\[
\essinf_{x\in\Omega} \rho(x)p(x)=\essinf_{x\in\Omega} s(\rho(x),x)+s^*(p(x),x)\geq \inf_{(z,x)\in\RR\times\Omega} s(z,x)>-\infty,
\]
where we used Lemma \ref{lem:s_star_increasing} to deduce that $s^*(p(x),x)\geq 0$.
Hence $\rho(x)p(x)$ is bounded from below, allowing us to conclude that $\rho p\in L^{\infty}(\Omega)$.
\end{proof}

\section{Properties of the Primal and the Dual Problems}\label{sec:dual}

In this section we show the equivalence of the primal and dual problems and give characterization of the corresponding extremizers.
We always assume $s$ satisfies \ref{assumption: convexity}-\ref{assumption: effective domain} and $c$ satisfies \ref{assumption: symmetry of c}-\ref{assumption: twist condition}.

\begin{prop}\label{prop:primal_dual}
Suppose that $\bar{\rho}\in X$ and satisfies \eqref{eq:good_data0}.
Then the primal problem \eqref{eq:primal_problem} has a unique minimizer $\rho^*\in X$ also satisfying \eqref{eq:good_data0}, and the dual problem \eqref{eq:dual_problem} admits at least one $c$-concave maximizer in $X^*$.
For the minimizing $\rho^*$ and any $c$-concave maximizer $p^*\in X^*$, 
we have
\[
J(\rho^*,\bar{\rho}) = \inf_{\rho\in X} J(\rho,\bar{\rho})=\sup_{p\in X^*} J^*(p,\bar{\rho}) = J^*(p^*,\bar{\rho})
\]
and
\[
\rho^*\in \partial  s^*(p^*(x),x)\quad a.e.\;x\in\Omega.
\]
Furthermore,
$T_{p^*}$ is the optimal map transporting $\bar{\rho}$ to $\rho^*$ for the cost $c$, and $\nabla p^*$ is unique $\rho^*$-a.e..
\end{prop}
\begin{remark}
If $\bar{\rho}\in X$ while \eqref{eq:good_data0} fails, one can check that $\bar{\rho}$ is itself the minimizer of the primal problem. Thus, the excluded cases are trivial.
\end{remark}
\begin{remark}
Uniqueness of the optimal pressure $p^*$ may fail when $s^*(\cdot,x)$ is not strictly convex. Nevertheless, we will show later that there always exists a largest and a smallest $c$-concave maximizing pressure among all $c$-concave maximizers of the dual energy. See Lemma \ref{lem:big_max_unique} and Lemma \ref{lem:pressure_special}, respectively.
\end{remark}
\begin{proof}
 From Lemma \ref{lem:s_star_increasing}, it follows that the dual energy $E^*(p)$ is monotone, i.e.\;if $p_0(x)\leq p_1(x)$ for a.e.\;$x\in\Omega$ then $E^*(p_0)\leq E^*(p_1)$. Thus, given some $p\in X^*$, we see from Lemma \ref{lem:ccc} that
 \[
 \int_{\Omega} p^{c\bar{c}c}(y)\bar{\rho}(y)\, dy-E^*(p^{c\bar{c}})\geq \int_{\Omega}  \int_{\Omega} p^{c}(y)\bar{\rho}(y)\, dy-E^*(p).
 \]
Hence,
\[
\sup_{p\in X^*,\, p^{c\bar{c}}=p} J^*(p,\bar{\rho})=\sup_{p\in X^*} J^*(p,\bar{\rho}),
\]
and so we can restrict our search to the space of $c$-concave functions.

Let $p_k$ be a sequence of $c$-concave functions such that
\[
 \lim_{k\to\infty}   J^*(p_k,\bar{\rho})=\sup_{p\in X^*, \,p^{c\bar{c}}=p}  J^*(p,\bar{\rho}).
\]
If we set $\alpha_k=\frac{1}{|\Omega|}\int_{\Omega} p_k(x)\, dx$, then $\tilde{p}_k=p_k-\alpha_k$ is $c$-concave and has zero mean.  Thanks to Lemma \ref{lem:lip}, $\tilde{p}_k$ are uniformly bounded in $W^{1,\infty}(\Omega)$.  So we can assume without loss of generality that $\tilde{p}_k$ converges uniformly to a function $\tilde{p}$ of mean zero.
Next, we choose
\[
\beta_k\in \argmax_{\beta\in \RR} \int_{\Omega} (\tilde{p}_k(x)+\beta)^c\bar{\rho}(x)\, dx-E^*(\tilde{p}_k+\beta).
\]
Since $(\tilde{p}_k(x)+\beta)^c=\tilde{p}_k^c(x)+\beta$, we see that $\beta_k$ must satisfy
\[
\int_{\Omega} \bar{\rho}(x)\, dx = \int_{\Omega}\zeta(x)\, dx, \quad \mbox{ for some } \zeta(x)\in \partial s^*(\tilde{p}_k(x)+\beta_k,x) \mbox{ for a.e.}\;x.
\]
Since $\partial s^*(\tilde{p}_k(x)+\beta,x)$ is increasing with respect to $\beta$, and $\{\tilde{p}_k\}_k$ are uniformly bounded, it follows from \eqref{eq:good_data0} and Lemma \ref{lem:s_star_increasing} that the sequence $\{\beta_k\}_{k=1}^\infty$ exists and is bounded uniformly in $\RR$.  Hence, we can assume without loss of generality that the $\beta_k$ converge to a finite limit $\tilde{\beta}$.

Define $p^*=(\tilde{p}+\tilde{\beta})^{c\bar{c}}$.  We then have the string of inequalities
\[
J^*(p^*,\bar{\rho})\geq  J^*(\tilde{p}+\tilde{\beta},\bar{\rho})\geq \limsup_{k\to\infty}  J^*(\tilde{p}_k+\beta_k,\bar{\rho}),
\]
where the last inequality follows from the fact that the $c$-transform and $-E^*$ are upper semi-continuous with respect to pointwise convergence.
Thanks to the choice of $\beta_k$, we see that
\[
\limsup_{k\to\infty}  J^*(\tilde{p}_k+\beta_k,\bar{\rho})\geq \limsup_{k\to\infty}  J^*(\tilde{p}_k+\alpha_k,\bar{\rho})=\sup_{p\in X^*, \, p^{c\bar{c}}=p}  J^*(p,\bar{\rho}).
\]
Therefore, we can conclude that $p^*$ is a $c$-concave maximizer of the dual problem.

In the rest of the proof, with abuse of the notation, we let $p^*$ be an arbitrary $c$-concave maximizer in $X^*$ of the dual problem.
Using Lemma \ref{lem:c_transform_variation}, the optimality condition for the dual problem at $p^*$ implies that there exists $\eta(x)\in  \partial s^*(p^*(x),x)$ such that
\[
\int_{\Omega} \phi(T_{p^*}(y))\bar{\rho}(y)\, dy-\int_{\Omega}\phi(x)\eta(x) dx= 0,
\]
for every continuous function $\phi:\Omega\to\RR$.
Thus, if we define $\rho^*:=T_{p^*\#}\bar{\rho}$, we must have
\[
\rho^*(x)\in \partial s^*(p^*(x),x)\quad\mbox{ for a.e.\;}x\in \Omega.
\]
Therefore, from Lemma \ref{lem:dual_relation} we have the duality relation
\[
\int_{\Omega} \rho^*(x)p^*(x)\, dx=E(\rho^*)+E^*(p^*).
\]
Hence,
\[
\int_{\Omega} (p^*)^{c}(y)\bar{\rho}(y)\, dy-E^*(p^*)=E(\rho^*)+\int_{\Omega} (p^*)^c(y)\bar{\rho}(y)\, dy-\int_{\Omega} p^*(x)\rho^*(x)\, dx= E(\rho^*)+C(\rho^*,\bar{\rho}),
\]
where the last equality follows from Theorem \ref{thm:fund_ot}.  This allows us to conclude that
\[
J^*(p^*,\bar{\rho})= J(\rho^*,\bar{\rho}). %\geq \inf_{\rho\in X} J(\rho, \bar{\rho}).
\]

On the other hand, if we dualize the energy in the primal problem \eqref{eq:primal_problem}, we get
\[
\inf_{\rho\in X} J(\rho,\bar{\rho})\geq \inf_{\rho\in X}\sup_{p\in X^*} \int_{\Omega} p(x)\rho(x)\, dx+C(\rho,\bar{\rho})-E^*(p).
\]
Interchanging the supremum and the infimum, it follows that
\[
\inf_{\rho\in X} J(\rho,\bar{\rho})\geq \sup_{p\in X^*}\inf_{\rho\in X} \int_{\Omega} p(x)\rho(x)\, dx+C(\rho,\bar{\rho})-E^*(p)=\sup_{p\in X^*}\int_{\Omega} p^c(x)\bar{\rho}(x)\, dx-E^*(p),
\]
where we used Lemma \ref{lem:c_duality} in the last equality.  Note that the last expression is nothing but $\sup_{p\in X^*}J^*(p,\bar{\rho})$.
Thus, %we have established that
\[
\inf_{\rho\in X} J(\rho,\bar{\rho})\geq \sup_{p\in X^*}J^*(p,\bar{\rho})= J^*(p^*,\bar{\rho})=J(\rho^*,\bar{\rho})
\]
Therefore, $\rho^*$ is a minimizer of the primal problem.  The cost functional $\rho\mapsto C(\rho,\bar{\rho})$ is strictly convex over $L^1(\Omega)$ \cite{OTAM}, so $\rho^*$ must be unique.
That $\rho^*$ satisfies \eqref{eq:good_data0} is obvious since it has the same total mass as $\bar{\rho}$.

Finally, if $p^*$ and $\tilde{p}^*$ are two maximizers of the dual problem, then both $T_{\tilde{p}^*}^{-1}$ and $T_{p^*}^{-1}$ are the optimal maps transporting $\rho^*$ to $\bar{\rho}$.  Thanks to Theorem \ref{thm:fund_ot}, the optimal map is unique up to a $\rho^*$-measure-zero set. It then follows from \eqref{eq:backward_map} that $\nabla \tilde{p}^*(x)=\nabla p^*(x)$ for $\rho^*$-a.e.\;$x\in \Omega$.
\end{proof}

%\newpage
To better understand the maximizing $c$-concave pressures, which may not be unique in certain situations, we denote the set of all $c$-concave maximizers of the dual functional to be
\begin{equation}\label{eqn: def of Sigma_rho}
\Sigma(\bar{\rho}):=\argmax_{p\in X^*,\, p^{c\bar{c}}=p} J^*(p,\bar{\rho}).
\end{equation}
The following lemma states that there always exists a largest $c$-concave maximizer in $\Sigma(\bar{\rho})$.

\begin{lemma} \label{lem:big_max_unique}
If $\bar{\rho}\in X$ and satisfies \eqref{eq:good_data0}, then there exists a unique $p^*\in \Sigma(\bar{\rho})$
such that $p^*\geq \tilde{p}$ for any $\tilde{p}\in \Sigma(\bar{\rho})$.
In other words, $p^*$ is the largest $c$-concave maximizer of the dual problem.
\end{lemma}
\begin{proof}
Let us first show that if $p_0, p_1\in \Sigma(\bar{\rho})$, then $q:=\max(p_0,p_1)\in \Sigma(\bar{\rho})$.  Clearly,
\[
q^{c\bar{c}}\geq \max(p^{c\bar{c}}_0, p_1^{c\bar{c}})=\max(p_0,p_1)=q,
\]
so $q$ is $c$-concave.
Next, let $\rho^*=\argmin_{\rho} J(\rho,\bar{\rho})$, which is unique. Then Proposition \ref{prop:primal_dual} together with Lemma \ref{lem:dual_relation} implies that $p_0(x), p_1(x)\in \partial s(\rho^*(x),x)$ for almost all $x\in \Omega$.  Therefore, we also have $q(x)\in \partial s(\rho^*(x),x)$ for almost all $x$.

Proposition \ref{prop:primal_dual} and Theorem \ref{thm:fund_ot} also imply $T_{p_0}(x)=T_{p_1}(x)$ for $\bar{\rho}$-a.e.\;$x$. We can then see that for $\bar{\rho}$-a.e.\;$x$,
\[
q\big(T_{p_0}(x)\big)+c\big(x,T_{p_0}(x)\big)=\max_{i\in \{0,1\}} \Big( p_i\big(T_{p_i}(x)\big)+c\big(x, T_{p_i}(x)\big)\Big).
\]
On the other hand, for $\bar{\rho}$ almost every $x\in \Omega$ and any $y\in\Omega$,
\[
\max_{i\in \{0,1\}} \Big( p_i\big(T_{p_i}(x)\big)+c\big(x, T_{p_i}(x)\big)\Big)\leq \max_{i\in \{0,1\}} \Big( p_i\big(y\big)+c\big(x,y\big)\Big)=q(y)+c(x,y).
\]
Hence,
\[
T_{p_0}(x)=\argmin_{y\in\Omega} q(y)+c(x,y)=T_{q}(x)
\]
for $\bar{\rho}$ almost every $x$.  Combining the facts that $q$ is $c$-concave, $q(x)\in \partial s(\rho^*(x),x)$ a.e.\;and $T_{q\,\#}\bar{\rho}=\rho^*$, one can follow the proof of Proposition \ref{prop:primal_dual} to show that
\[
J^*(q,\bar{\rho})= \inf_{\rho} J(\rho, \bar{\rho}).
\]
This implies that $q$ is a maximizer.

Now define
\[
p^*(x)=\sup_{\tilde{p}\in \Sigma(\bar{\rho})} \tilde{p}(x).
\]
Clearly $p^*\geq \tilde{p}$ for any $\tilde{p}\in \Sigma(\bar{\rho})$, and thus it suffices to show that $p^*$ is a $c$-concave maximizer.
The $c$-concavity of $p^*$ is clear, since
\[
(p^*)^{c\bar{c}}(x)\geq \sup_{\tilde{p}\in \Sigma(\bar{\rho})}\tilde{p}^{c\bar{c}}(x)=\sup_{\tilde{p}\in \Sigma(\bar{\rho})}\tilde{p}(x)=p^*(x).
\]

Let $\{x_k\}_{k\in\ZZ_+}$ be a dense subset of $\Omega$.  For each $ n,k\in\ZZ_+$, there exists some $p_{n,k}\in \Sigma(\bar{\rho})$ such that $p^*(x_k)\leq p_{n,k}(x_k)+\frac{1}{n}$.  For each $N\in\ZZ_+$, define
\[
q_N:=\max_{1\leq k\leq N} p_{N,k}.
\]
Our argument above shows that $q_N\in \Sigma(\bar{\rho})$ for all $N\in \ZZ_+$. The $c$-concavity implies that $p^*$ and the family $\{q_N\}_{N\in\ZZ_+}$ are uniformly Lipschitz.
Hence, for any $x\in\Omega$ we have
\[
|p^*(x)-q_N(x)|\leq \frac{1}{N}+2L\min_{1\leq k\leq N} |x-x_k|,
\]
where $L$ is the Lipschitz constant associated to $c$-concave functions on $\Omega$.
Thus,
\[
\lim_{N\to\infty} |p^*(x)-q_N(x)|=0.
\]
Moreover, since
\[
\sup_{x\in \Omega}\min_{1\leq k\leq N} |x-x_k| \to 0\quad \mbox{ as } N\to +\infty
\]
due to the density of $\{x_k\}_{k\in \ZZ_+}$, we know that the convergence from $q_N$ to $p^*$ is uniform in $x$.
%
%Since $p^*$ and $q_N$ are uniformly Lipschitz, the convergence is in fact uniform in $x$.
The functional $J^*(\cdot, \bar{\rho})$ is clearly continuous with respect to uniform convergence, therefore $p^*$ is a maximizer.
\end{proof}

Next we prove a maximum-principle-type result for the pressure variable, which is of independent interest.
In order for the statement to hold, we will need to assume the initial density is almost everywhere positive.

\begin{prop}\label{prop:maximum_principle}
Suppose that $\bar{\rho}\in X$ satisfies \eqref{eq:good_data0} and that $\bar{\rho}>0$ a.e.\;in $\Omega$.
Let $\bar{p}(x)\in \partial s(\bar{\rho}(x),x)$ for almost every $x$.
Denote
$$
a:=\essinf_{x\in\Omega} \bar{p}(x),\quad b:=\esssup_{x\in\Omega} \bar{p}(x).
$$
%\[
Let $\rho^*$ be the minimizer of the primal problem and take an arbitrary $\tilde{p}\in \Sigma(\bar{\rho})$. %\argmax_{p\in X^*, \, p^{c\bar{c}}=p} J^*(p,\bar{\rho})
%\]
We have the following dichotomy:
\begin{enumerate}
\item If $\rho^*\neq\bar{\rho}$, then $\tilde{p}(x)\in [a,b]$ for all $x$.
\item If $\rho^*=\bar{\rho}$, then all members of $\Sigma(\bar{\rho})$ are constant functions.
More precisely, there exists a bounded closed interval $[a',b']\subset \RR$ with $[a,b]\cap [a',b'] \neq \varnothing$, such that $p\in \Sigma(\bar{\rho})$ if and only if $p$ is constant function with its value in $[a',b']$.
\end{enumerate}
\end{prop}

\begin{proof}
Choose a $c$-concave maximizer $\tilde{p}\in \Sigma(\bar{\rho})%\argmax_p J^*(p,\bar{\rho})
$ and let $\rho^*=T_{\tilde{p}\,\#}\bar{\rho}$.
It follows from Proposition \ref{prop:primal_dual} that $\rho^*$ is the unique solution of the primal problem and $\rho^*\in \partial s^*(\tilde{p},x)$ a.e..
Let
\begin{align*}
V&:=\{x\in\Omega: \tilde{p}(x)>b\},\\
U&:= \{x\in\Omega: \tilde{p}(x)<a\}.
\end{align*}
Since $\tilde{p}$ is $c$-concave and hence Lipschitz, both $V$ and $U$ are open sets.

Suppose that $U\cup V\neq\varnothing$.  $V$ is an upper level set of $\tilde{p}$ and $U$ is a lower level set of $\tilde{p}$, so we have the inclusions $T_{\tilde{p}}^{-1}(V)\subset V$ and $U\subset T_{\tilde{p}}^{-1}(U)$.  The monotonicity of $\partial s^*(\cdot,x)$ implies that
  \[
 \bar{\rho}(x)\leq
 %\partial_p s^*(\bar{p},x)\leq \partial_p s^*(\tilde{p}(x),x)=
 \rho^*(x) \quad \textrm{for a.e.}\; x\in V,
   \]
   and
    \[
   \rho^*(x)\leq
   %=\partial_p s^*(\tilde{p}(x),x)\leq \partial_p s^*(\bar{p},x)=
   \bar{\rho}(x)  \quad \textrm{for a.e.}\; x\in U.
   \]
Here the monotonicity of $\partial s^*(\cdot,x)$ is understood in the following sense: for any $b^{(0)}<b^{(1)}$ and any $\eta^{(i)}\in \partial s^*(b^{(i)},x)$ $(i = 0,1)$, we have $\eta^{(0)}\leq \eta^{(1)}$.
Hence, we can compute
\[
0\leq \int_{V} \rho^*(x)-\bar{\rho}(x)\, dx=-\int_{V-T_{\tilde{p}}^{-1}(V)} \bar{\rho}(x)\, dx \leq 0,
\]
and
\[
0\geq \int_{U} \rho^*(x)-\bar{\rho}(x) \,dx=\int_{T_{\tilde{p}}^{-1}(U)-U} \bar{\rho}(x)\, dx \geq 0.
\]
Hence it follows that $\rho^*(x)=\bar{\rho}(x)$ for almost all $x\in U\cup V$.

Consider the maps
\[
S_1(x)=\begin{cases}
x & \textrm{if}\; x\in V,\\
T_{\tilde{p}}(x) & \textrm{otherwise,}
\end{cases}
\]
and
\[
S_2(x)=\begin{cases}
x & \textrm{if}\; x\in U,\\
T^{-1}_{\tilde{p}}(x) & \textrm{otherwise.}
\end{cases}
\]
Since $T_{\tilde{p}}^{-1}(V)\subset V$ and $\bar{\rho}(x)=\rho^*(x)$ for a.e.\;$x\in V$,  it follows that $\big(S_{1 \#}\bar{\rho}\big)(x)=\rho^*(x)$ a.e.\;on $V$ and $\big(S_{1 \#}\bar{\rho}\big)(x)\leq \rho^*(x)$ for a.e.\;$x\notin V$.  But since pushforwards preserve mass, this is only possible if $S_{1 \#}\bar{\rho}=\rho^*$.  Next, it is clear that the transportation cost of $S_1$ cannot exceed the transportation cost of $T_{\tilde{p}}$, thus $S_1$ must be an optimal transport map between $\bar{\rho}$ and $\rho^*$.  The uniqueness of optimal maps (see Theorem \ref{thm:fund_ot}) then implies that $S_1=T_{\tilde{p}}$ $\bar{\rho}$-almost everywhere.  An analogous argument shows that $S_2=T_{\tilde{p}}^{-1}$ $\rho^*$-almost everywhere.
Now we can conclude that $T_{\tilde{p}}(x)=T_{\tilde{p}}^{-1}(x)=x$ for $\bar{\rho}$-almost all $x\in U\cup V$.
Combining \eqref{eq:forward_map} and the assumptions \ref{assumption: symmetry of c}-\ref{assumption: regularity of c} we can conclude that $\nabla \tilde{p}(x)=0$ for $\bar{\rho}$-a.e.\;$x\in U\cup V$.
Since $\bar{\rho}>0$ a.e.\;in $\Omega$, this implies $\nabla \tilde{p}(x)=0$ a.e.\;in $U\cup V$.

Given some $x_0\in U\cup V$, let $A$ be the connected component of $x_0$ in $U\cup V$. It is then clear that $\tilde{p}$ is constant on $A$.  Since $\tilde{p}$ is $c$-concave and hence Lipschitz, $A$ must be both open and closed.  By the connectivity of $\Omega$, this is only possible if $A=\Omega$.
It then follows that $\tilde{p}$ is constant, $\nabla\tilde{p}=0$ on $\Omega$ and hence, $\rho^*=\bar{\rho}$ almost everywhere.
Therefore, if $\rho^*\neq \bar{\rho}$, we must have $U\cup V = \varnothing$ and thus $\tilde{p}\in [a,b]$.

Now we prove the second part of the dichotomy.
If $\rho^*=\bar{\rho}$, then for any $p\in \Sigma(\bar{\rho})$  we have $T_{p}(x)=x$ for  $\bar{\rho}$ almost every $x$.  Hence, $\Sigma(\bar{\rho})$ only contains constant functions.  Now it is clear from Proposition \ref{prop:primal_dual} that $p\in \Sigma(\bar{\rho})$ if and only if $p$ is constant and $p\in \partial s(\rho^*(x),x)$ for almost every $x$.    Combining the condition $p\in \partial s(\rho^*(x),x)$ with the assumption \eqref{eq:good_data0}, we see that the maximizers must be contained in a bounded set.       $J^*$ is concave and continuous with respect to uniform convergence, so the set of maximizers must be closed and convex.   Now it follows that $\Sigma(\bar{\rho})=[a', b']$ for some $a', b'\in \RR$.

Finally,  if $[a',b']\cap [a,b]=\varnothing$, then there must exist a set $\Omega'\subset\Omega$ of positive measure such that $[a,b]\cap \partial s(\rho^*(x),x)=\varnothing$ for all $x\in\Omega'$.
However, this contradicts the existence of $\bar{p}$, and thus, $[a,b]\cap [a',b']\neq \varnothing$.
\end{proof}

\section {$L^1$-contraction}
\label{sec: L^1 contraction}

In this section we prove the $L^1$-contraction principle for the discrete solutions.   At the discrete level, we have the $c$-concavity of the pressure functions, which provides regularity that is independent of $s$. This allows for a pointwise argument that is not as dependent on the regularity of $s$ as it is for the continuum solutions.

\medskip

At the heart of the $L^1$-contraction principle is the following simple observation.
\begin{lemma}\label{lem:fundamental_property}
Let $p_0, p_1:\Omega\to\RR$ be $c$-concave functions and let $U=\{x\in\Omega: p_0(x)<p_1(x)\}$.  If $T_{p_1}(y)\in U$, then $T_{p_0}(y)\in U$.
\end{lemma}
\begin{proof}
Recall that
\[
T_{p_i}(y)=\argmin_{x\in\Omega} p_i(x)+c(x,y).
\]
For any $\tilde{x}\notin U$ and $y$ satisfying $T_{p_1}(y)\in U$,
\[
\begin{split}
p_0(\tilde{x})+c(\tilde{x},y)\geq &\;p_1(\tilde{x})+c(\tilde{x},y)\\
\geq &\;
p_1\big(T_{p_1}(y)\big)+c\big(T_{p_1}(y),y\big)> p_0\big(T_{p_1}(y)\big)+c\big(T_{p_1}(y),y\big),
\end{split}
\]
where the second inequality follows from the definition of $T_{p_1}(y)$. %, and the first and third inequalities follow from the definition of $U$.
The above computation shows us that compared to $\tilde{x}\notin U$, $T_{p_1}(y)$ is always a better competitor for $\min_{x\in\Omega} p_0(x)+c(x,y)$.  Thus, it follows that $T_{p_0}(y)\in U$.
\end{proof}

We first establish the $L^1$-contraction  property in the case that the mapping $z\mapsto \partial s(z,x)$ is strictly monotone.
\begin{lemma}\label{lem:l1_contraction}
Given $\rho_0, \rho_1\in X$ satisfying \eqref{eq:good_data0}, let
\[
\rho_i^*=\argmin_{\rho\in X} J(\rho,\rho_i). %\int_{\Omega} s(\rho(x),x)\, dx+C(\rho, \rho_i).
\]
If for all $x\in \Omega$, $z\mapsto \partial s(z,x)$ is strictly increasing when $z\geq 0$, then
\[
\norm{(\rho_1^*-\rho_0^*)_+}_{L^1(\Omega)}\leq \norm{(\rho_1-\rho_0)_+}_{L^1(\Omega)}.
\]
\end{lemma}
\begin{remark}
Here the strict monotonicity of $z\mapsto \partial s(z,x)$ is understood as follows: for any $0\leq z^{(0)}<z^{(1)}$ and any $q^{(i)}\in \partial s(z^{(i)},x)$ $(i = 0,1)$, we have $q^{(0)}< q^{(1)}$.
\end{remark}
\begin{proof}
By Proposition \ref{prop:primal_dual}, we can choose
%\[
$p_i\in \Sigma(\rho_i)$, %\argmax_{p\in X^*,\,p^{c\bar{c}}= p} J^*(p,\rho_i). %\int_{\Omega} \rho_i(x) p^c(x)-s^*(p(x),x)\, dx.
%\]
where $\Sigma(\rho_i)$ is the set of $c$-concave maximizers for the data $\rho_i$ as defined in \eqref{eqn: def of Sigma_rho}.
Then $\rho_i^*=T_{p_i\#}\rho_i$ and $p_i(x)\in \partial s(\rho_i^*,x)$.  Since $z\mapsto \partial s(z,x)$ is strictly increasing for $z\geq 0$,
\[
\sgn(p_1(x)-p_0(x))_+=\sgn(\rho_1^*(x)-\rho_0^*(x))_+
\] whenever $\rho_1^*(x)\neq \rho_0^*(x)$.
Therefore,
\[
\int_{\Omega} (\rho_1^*(x)-\rho_0^*(x))_+\,dx=\int_{\Omega} (\rho_1^*(x)-\rho^*_0(x))\sgn(p_1(x)-p_0(x))_+\, dx.
\]
Write $\vp(x):=\sgn(p_1(x)-p_0(x))_+$ and note that $\vp$ is simply the characteristic function of the set $U=\{x\in\Omega: p_0(x)<p_1(x)\}$.
Using $\rho_i^*=T_{p_i\#}\rho_i$, the previous line becomes
\[
\int_{\Omega} (\rho_1^*(x)-\rho_0^*(x))_+\,dx=\int_{\Omega} (\rho_1(x)-\rho_0(x))\vp(T_{p_1}(x))+\rho_0(x)(\vp(T_{p_1}(x))-\vp(T_{p_0}(x)))\, dx.
\]
Since $\vp\in \{0,1\}$,
\[
\int_{\Omega} (\rho_1(x)-\rho_0(x))\vp(T_{p_1}(x))\, dx\leq \norm{(\rho_1-\rho_0)_+}_{L^1(\Omega)}.
\]
On the other hand, Lemma \ref{lem:fundamental_property} gives $\vp(T_{p_1}(x))-\vp(T_{p_0}(x))\leq 0$, so the result follows.
\end{proof}

To prove the $L^1$-contraction we shall remove the strict monotonicity assumption of $\partial s(\cdot,x)$ by approximation.

\begin{theorem}\label{thm:l1_contraction}
Suppose $s$ satisfies \ref{assumption: convexity}-\ref{assumption: effective domain} and $c$ satisfies \ref{assumption: symmetry of c}-\ref{assumption: twist condition}. Let $\rho_0, \rho_1\in X$ satisfy \eqref{eq:good_data0}, and
\[
\rho_i^*=\argmin_{\rho\in X} J(\rho,\rho_i),
\]
where $J(\rho,\rho_i)$ is defined in \eqref{eq:primal_problem}.
Then we have
\[
\norm{(\rho_1^*-\rho_0^*)_+}_{L^1(\Omega)}\leq \norm{(\rho_1-\rho_0)_+}_{L^1(\Omega)}.
\]
\end{theorem}

\begin{proof}

Define
\[
s_{\delta}(z,x)=s(z,x)+\delta\big(\sqrt{1+z^2}-1\big).
\]
Obviously, $s_\delta$ satisfies \ref{assumption: convexity} and \ref{assumption: effective domain}.
Let
\[
E_{\delta}(\rho)=\int_{\Omega} s_{\delta}(\rho(x),x)\, dx.
\]
By \ref{assumption: effective domain}, $X_\delta :=\{\rho\in L^1(\Omega): E_\delta(\rho)<\infty\}$ coincides with $X$.
To verify that $\rho_i$ satisfies \eqref{eq:good_data0} with $s^*$ replaced by $s_\delta^*$, it suffices to show that for any $x\in \Omega$ and $b\in \mathbb{R}_+$,
\begin{equation}\label{eqn: sufficient condition rho_i satisfies the non-saturation condition for s_delta}
\partial s^*(b-\delta,x)\leq \partial s^*_\delta(b,x)\leq \partial s^*(b,x).
\end{equation}
Here by writing inequalities between these subdifferentials, we mean any choice of elements in these sets satisfies this inequality.
Indeed, by Lemma \ref{lem:dual_relation}, $b\in \partial s_\delta \big(\partial s_\delta^*(b,x),x\big)$.
By the way $s_\delta$ is defined, we find
\[
b-\delta \leq \sup\partial s \big(\partial s_\delta^*(b,x),x\big)\quad \mbox{ and }\quad \inf\partial s \big(\partial s_\delta^*(b,x),x\big)\leq b,
\]
which implies \eqref{eqn: sufficient condition rho_i satisfies the non-saturation condition for s_delta} by applying Lemma \ref{lem:dual_relation} again.

Let
\[
\rho_{i,\delta}^*=\argmin_{\rho\in X} E_{\delta}(\rho)+C(\rho,\rho_i).
\]
Now thanks to Lemma \ref{lem:l1_contraction}, we have
\[
\norm{(\rho_{1,\delta}^*-\rho_{0,\delta}^*)_+}_{L^1(\Omega)}\leq \norm{(\rho_1-\rho_0)_+}_{L^1(\Omega)}.
\]
For $i=0,1$, the family $\{\rho_{i,\delta}^*\}_{\delta>0}$ lies in a bounded subset of $X$, in the sense that
$$
E(\rho_{i,\delta}^*)\leq E_\delta(\rho_i)\leq E(\rho_i)+\delta\|\rho_i\|_{L^1(\Omega)}<+\infty.
$$
Thanks to de la Vall\'{e}e-Poussin's theorem on uniform integrability and conditions \ref{assumption: convexity} and \ref{assumption: effective domain} on the energy \cite{meyer_book},  the family $\rho_{i, \delta}^*$ is weakly compact in $L^1$.  Hence, without loss of generality, we can assume that $\rho_{i, \delta}^*$ converges weakly in $L^1(\Omega)$ to some limit $\tilde{\rho}_i\in L^1(\Omega)$.

  From Theorem \ref{thm:fund_ot}, it follows that the optimal transport cost $\rho\mapsto C(\rho,\rho_i)$ is lower semi-continuous with respect to $L^1(\Omega)$ weak convergence.  Furthermore, the convexity of $s$ can be used to show that $E$ is $L^1$-weakly lower semi-continuous (see for instance \cite[\S 2.B, Theorem 1]{Eva90} for the proof of a similar argument).  Therefore,
  \[
  \int_{\Omega} s(\tilde{\rho}_i(x),x)\, dx+C(\tilde{\rho}_i, \rho_i)\leq  \liminf_{\delta\to 0} \int_{\Omega} s(\rho^*_{i,\delta}(x),x)\, dx+C(\rho^*_{i,\delta}, \rho_i).
  \]
It is then clear that for any $\rho\in L^1(\Omega)$,
\[
\liminf_{\delta\to 0} \int_{\Omega} s(\rho^*_{i,\delta}(x),x)\, dx+C(\rho^*_{i,\delta}, \rho_i)\leq \liminf_{\delta\to 0} \int_{\Omega} s_{\delta}(\rho(x),x)\, dx+C(\rho, \rho_i).
\]
Since
\[
\lim_{\delta\to 0} \int_{\Omega} s_{\delta}(\rho(x),x)\, dx =\int_{\Omega} s(\rho(x),x)\, dx,
\]
we take infimum over all $\rho\in L^1(\Omega)$ and find that
\[
\int_{\Omega} s(\tilde{\rho}_i(x),x)\, dx+C(\tilde{\rho}_i, \rho_i)\leq \inf_{\rho\in L^1(\Omega)}\int_{\Omega} s(\rho(x),x)\, dx+C(\rho, \rho_i).
\]
So $\tilde{\rho}_i$ is a minimizer, and obviously $\tilde{\rho}_i\in X$.
The uniqueness of the minimizers established in Proposition \ref{prop:primal_dual} allows us to conclude that $\tilde{\rho}_i=\rho_i^*$ a.e..
Finally, %the weak lower semi-continuity of the $L^1$-norm implies that
\[
\begin{split}
\norm{(\rho_{1}^*-\rho_{0}^*)_+}_{L^1(\Omega)} = &\;\lim_{\delta\to 0} \int_\Omega (\rho_{1,\delta}^*-\rho_{0,\delta}^*)\chi_{\{\rho_1^*>\rho_0^*\}}(x)\,dx\\
\leq &\;\liminf_{\delta\to 0}\norm{(\rho_{1,\delta}^*-\rho_{0,\delta}^*)_+}_{L^1(\Omega)}\leq \norm{(\rho_1-\rho_0)_+}_{L^1(\Omega)}.
\end{split}
\]
\end{proof}

In Lemma \ref{lem:big_max_unique}, we have shown that there always exists the largest $c$-concave maximizing pressure of the dual problem.
The following lemma states that we can always find the smallest $c$-concave maximizing pressure as well.

\begin{lemma}\label{lem:pressure_special}
If $\bar{\rho}\in X$ and satisfies \eqref{eq:good_data0}, then there exists a unique $p^*\in \Sigma (\bar{\rho})$ such that $p^*\leq \tilde{p}$ for any $\tilde{p}\in  \Sigma(\bar{\rho})$.
In other words, $p^*$ is the smallest $c$-concave maximizer of the dual problem.
%\begin{equation}\label{eqn: arbitrary maximizer}
%\tilde{p}\in\argmax_{p\in X^*,\, p^{c\bar{c}}=p}  J^*(p,\bar{\rho}).
%\end{equation}
\end{lemma}
\begin{proof}
For any $k\in\ZZ_+$, define $s_{k}^*(b,x):=s^*(b,x)+\frac{1}{k}\ln(1+e^b)$ and
\[
J^*_{k}(p,\bar{\rho}):=\int_{\Omega} p^c(x)\bar{\rho}(x)-s_{k}^*(p(x),x)\, dx.
\]
Let
$$
s_k(z,x):=\sup_{p\in \mathbb{R}}pz - s_k^*(p,x).
$$
It is clear that for all $k\in \ZZ_+$, $s_k(z,x)\leq s(z,x)$ satisfies \ref{assumption: convexity}.
We claim that $s_k(z,x)$ also satisfies \ref{assumption: effective domain}.
Indeed,  by Lemma \ref{lem:s_star_increasing}, $s^*(\cdot,x)$ is non-negative and increasing, and so is $s^*_k(\cdot,x)$.
For $z = 0$,
$$
s_k(0,x)=\sup_{p\in \mathbb{R}} - s_k^*(p,x) \leq -\inf_{p\in \mathbb{R}} s^*(p,x)-\inf_{p\in \mathbb{R}} \frac{1}{k}\ln(1+e^{p}) = s(0,x)=0.
$$
On the other hand,
\[
s_k(0,x)\geq \lim_{p\to -\infty} -s_k^*(p,x)=0.
\]
When $z<0$,
$$
\lim_{p\to -\infty} pz - s_k^*(p,x) = \lim_{p\to -\infty} pz + \lim_{p\to -\infty} -s_k^*(p,x)
=+\infty.
$$
Note that the second limit is $0$ thanks to the argument above for the case $z=0$.
For $z>0$,
\[
\begin{split}
s_k(z,x) \geq &\;\max\left\{\sup_{p\leq 0}pz - s_k^*(0,x),\;\sup_{p> 0}pz - s_k^*(p,x)\right\}\\
\geq &\;\max\left\{ - s^*(0,x)-\frac{\ln 2}{k},\;\sup_{p> 0}pz - s^*(p,x)-\frac{p+1}{k}\right\}\\
\geq  &\;\max\left\{ \inf_{(z,x)} s(z,x)-\frac{\ln 2}{k},\;\sup_{p> 0}p\left(z-\frac{1}{k}\right) - s^*(p,x)-\frac{1}{k}\right\}.
\end{split}
\]
The first term in the last line implies $\inf_{(z,x)}s_k(z,x)>-\infty$.
We also observe that, when $\tilde{z}:=z-\frac{1}{k}>0$ is sufficiently large, for all $x\in \Omega$,
$$
s(\tilde{z},x) = \sup_{p>0} p\tilde{z} -s^*(p,x).
$$
Indeed, by \ref{assumption: effective domain}, the left hand side is positive throughout $\Omega$ whenever $\tilde{z}$ is sufficiently large, while
\[
\sup_{p\leq 0} p\tilde{z} -s^*(p,x)\leq 0.
\]
Therefore, for $z\gg 1$,
\[
s_k(z,x)
\geq s\left(z-\frac{1}{k},x\right)-\frac{1}{k}.
\]
This shows $s_k(\cdot,x)$ has the uniform-in-$x$ superlinear growth in \ref{assumption: effective domain}.

Let
$$
E_k(\rho) =\int_{\Omega}s_k(\rho(x),x)\,dx.
$$
Then $E_k(\bar{\rho})\leq E(\bar{\rho})<+\infty$ since $s_k(z,x)\leq s(z,x)$.
Moreover,
\[
0<\int_{\Omega}\bar{\rho}(x)\,dx<\lim_{b\to\infty} \int_{\Omega} \sup\partial s^*_{k}(b,x)\, dx.
\]
Since $s_{k}^*(\cdot,x)$ is strictly convex, Proposition \ref{prop:primal_dual} implies that $J^*_{k}(p,\bar{\rho})$ has a unique $c$-concave maximizer $p_{k}^*$ for all $k\in\ZZ_+$.

Now we will show that $p_{k}^*$ is pointwise increasing with respect to $k$. Given $k_0<k_1$, let $U=\{x\in\Omega: p^*_{k_1}(x)< p^*_{k_0}(x)\}$ and let $\phi$ be the characteristic function of $U$.  The optimality of the $p_{k_i}^*$ implies that there exists $\eta_i(x)\in \partial s^*_{k_i}\big(p_{k_i}^*(x),x\big)$ such that
\[
\int_{\Omega} \phi\big(T_{p^*_{k_i}}(x)\big)\bar{\rho}(x)\, dx=\int_U \eta_i(x)\, dx.
\]
Hence, %Subtracting the above equation at $i=0$ from the equation at $i=1$, we get
\begin{equation}\label{eq:delta_compare}
\int_{\Omega} \Big(\phi\big(T_{p^*_{k_1}}(x)\big)-\phi\big(T_{p^*_{k_0}}(x)\big)\Big)\bar{\rho}(x)\, dx=\int_U (\eta_1(x)-\eta_0(x))\,dx.
\end{equation}
Thanks to Lemma \ref{lem:fundamental_property}, the left hand side of \eqref{eq:delta_compare} is nonnegative.  On the other hand, from the subdifferential condition $\eta_i(x)\in \partial s^*_{k_i}\big(p_{k_i}^*(x),x\big)$, it follows that $\eta_1(x)<\eta_0(x)$ for all $x\in U$.  Thus, \eqref{eq:delta_compare} can only hold if $U$ has measure zero.  Since the $p_{k_i}^*$ are Lipschitz, we can then conclude that $p^*_{k_0}\leq p^*_{k_1}$ everywhere.   Note an identical argument shows that $p_k^*\leq \tilde{p}$ for any $k\in \ZZ_+$ and any maximizing pressure $\tilde{p}\in \Sigma(\bar{\rho})$. % in \eqref{eqn: arbitrary maximizer}.

Let $p^*(x):=\lim_{k\to\infty} p_k^*(x)$, and we note that $p_k^*$ must converge uniformly to $p^*$ since they are uniformly Lipschitz (by Lemma \ref{lem:lip}) and uniformly bounded (see the proof of Proposition \ref{prop:primal_dual}).  We also have
 \[
(p^*)^{c\bar{c}}\geq \lim_{k\to\infty} (p_k^*)^{c\bar{c}}=\lim_{k\to\infty} p_k^*=p^*,
\]
so $p^*$ is $c$-concave.
Now take an arbitrary maximizing pressure $\tilde{p}\in  \Sigma(\bar{\rho})$, 
which is bounded and Lipschitz. 
Since the $c$-transform and $-E^*$ are upper semi-continuous with respect to the  pointwise convergence, we get the string of inequalities
  \[
J^*(p^*,\bar{\rho})\geq \limsup_{k\to\infty} J^*(p^*_k,\bar{\rho})\geq \limsup_{k\to\infty} J^*_k(p^*_k,\bar{\rho})\geq \lim_{k\to\infty}  J^*_k(\tilde{p},\bar{\rho})= J^*(\tilde{p},\bar{\rho}).
\]
Thus, $p^*\in  \Sigma(\bar{\rho})$.
That $p^*$ is the smallest possible $c$-concave maximizer follows from our earlier observation that $p_k^*\leq \tilde{p}$ and the pointwise convergence $p_k^*\to p^*$. % for any $k\in \ZZ_+$ and any $\tilde{p}\in\argmax_{p\in X^*, p^{c\bar{c}}=p}  J^*(p,\bar{\rho})$.
\end{proof}

Now we are ready to prove the discrete comparison principle.

\begin{theorem} \label{thm: direct comparison of rho and p}
Under the assumptions of Theorem~\ref{thm:l1_contraction}, suppose $\rho_0\leq \rho_1$ a.e.\;in $\Omega$.
Then
\begin{enumerate}
\item $\rho_0^*\leq \rho_1^*$ a.e.\;in $\Omega$.

\item
Let $\Sigma(\rho_i)$ be defined as in \eqref{eqn: def of Sigma_rho}.
For $i = 0,1$, let $p_i^+\in \Sigma(\rho_i)$ 
be the largest $c$-concave maximizers constructed in Lemma \ref{lem:big_max_unique}.
Then $p_0^+\leq p_1^+$ in $\Omega$.

\item
Alternatively, let $p_i^-\in \Sigma(\rho_i)$ be the smallest $c$-concave maximizers constructed in Lemma \ref{lem:pressure_special}.
Then we also have $p_0^-\leq p_1^-$ in $\Omega$.
\end{enumerate}
\end{theorem}
\begin{proof}
That $\rho_0^*\leq \rho_1^*$ a.e.\;in $\Omega$ is an immediate consequence of Theorem \ref{thm:l1_contraction}.

To show the inequality between $p_i^\pm$, we first note that by Lemma \ref{lem:dual_relation}, the subdifferential relation $\rho_i^*(x)\in \partial s^*(p_i^\pm(x),x)$ a.e.\;implies that $p^\pm_i(x)\in \partial s(\rho_i^*(x),x)$ a.e..

\medskip

We first prove $p_0^+\leq p_1^+$.
It suffices to show that
$$
\tilde{p}^+:=\max(p_0^+, p_1^+)\in \Sigma(\rho_1). %\argmax_{p\in X^*,\, p^{c\bar{c}}=p} J^*(p,\rho_1).
$$
Indeed, by the maximality of $p^+_1$, this implies $\tilde{p}^+=p^+_1$ and thus $p_0^+\leq p_1^+$.
Clearly, $\tilde{p}^+$ is $c$-concave since
\[
(\tilde{p}^+)^{c\bar{c}}\geq\max\big((p_0^+)^{c\bar{c}},(p_1^+)^{c\bar{c}}\big)=\tilde{p}^+.
\]
Let
\[
U:=\{x\in\Omega: p_1^+(x)<p_0^+(x)\}.
\]
Since $\partial s(\cdot,x)$ is increasing and $\rho_0^*\leq \rho_1^*$, we must have $\rho_0^*=\rho_1^*$ a.e.\;on $U$.   Hence, it follows that
\begin{equation}\label{dual}
\rho_1^*(x)\in \partial s^*(\tilde{p}^+(x),x)\quad \mbox{ a.e.\;}x\in \Omega.
\end{equation}
If we can show that
\begin{equation}\label{main000}
\tilde{\rho}_1:= T_{\tilde{p}^+\#}\rho_1 = \rho_1^*\quad \mbox{ a.e.\;on }\Omega,
\end{equation}
then together with \eqref{dual} it follows that any infinitesimal variation of $J^*(p,\rho_1)$ over $p$ at $\tilde{p}^+$ will not make its value increase, and we conclude from the concavity of the energy that $\tilde{p}^+$ is a maximizer, so $\tilde{p}^+\in \Sigma(\rho_1)$.

Using the same logic as Lemma \ref{lem:fundamental_property}, we can show that if $T_{p_0^+}(x)\in U$ then $T_{\tilde{p}^+}(x)\in U$.  Furthermore, since $\tilde{p}^+$ and $p_0^+$ agree on $U$, \eqref{eq:map_argmin} implies that $T_{\tilde{p}^+}(x)=T_{p_0^+}(x)$ when $T_{p_0^+}(x)\in U$.
Define $\phi$ to be the characteristic function of $U$.
Then for any non-negative smooth function $f:\Omega\to[0,\infty)$
\[
\begin{split}
\int_{U} \tilde{\rho}_1(x)f(x)\, dx=&\;\int_{\Omega} \rho_1(x)\phi(T_{\tilde{p}^+}(x))f(T_{\tilde{p}^+}(x))\, dx\\
\geq&\; \int_{\Omega} \rho_0(x)\phi(T_{p_0^+}(x))f(T_{p_0^+}(x))\, dx=\int_U \rho_0^*(x)f(x)\, dx.
\end{split}
\]
Hence, $\rho_1^* =\rho_0^* \leq \tilde{\rho}_1$ a.e.\;in $U$.

On the other hand, using Lemma \ref{lem:fundamental_property} again, we see that if $T_{p_1^+}(x)\notin U$, then $T_{\tilde{p}^+}(x)\notin U$.
This implies $\phi(T_{\tilde{p}^+}(x))\leq \phi(T_{p_1^+}(x))$.
Therefore,
\[
\int_U \tilde{\rho}_1(x)\, dx=\int_{\Omega} \rho_1(x)\phi\big(T_{\tilde{p}^+}(x)\big)\, dx\leq \int_{\Omega} \rho_1(x)\phi\big(T_{p_1^+}(x)\big)\, dx=\int_U \rho_1^*(x)\,dx.
\]
Hence $\tilde{\rho}_1(U)\leq \rho_1^*(U)$. Combining this with the above fact that $\rho_1^*(x)\leq \tilde{\rho}_1(x)$ for a.e.\;$x\in U$, we must have $\tilde{\rho}_1=\rho_1^*$ a.e.\;on $U$.

Using the same logic from above, we must have $T_{\tilde{p}^+}(x)=T_{p_1^+}(x)$ when $T_{p_1^+}(x)\notin U$, since $\tilde{p}^+=p_1^+$ on $\Omega\setminus U$.  Therefore,
\[
\begin{split}
\int_{\Omega\setminus U} \tilde{\rho}_1(x)f(x)\, dx=&\;\int_{\Omega} \big(1-\phi\big(T_{\tilde{p}^+}(x)\big)\big)\rho_1(x)f\big(T_{\tilde{p}^+}(x)\big)\, dx\\
\geq &\;\int_{\Omega} \big(1-\phi\big(T_{p_1^+}(x)\big)\big)\rho_1(x)f\big(T_{p_1^+}(x)\big)\, dx=\int_{\Omega\setminus U} \rho_1^*(x)f(x)\, dx.
\end{split}
\]
Hence, $\tilde{\rho}_1\geq \rho_1^*$ a.e.\;in $\Omega\setminus U$.   Since $\tilde{\rho}_1= \rho_1^*$ a.e.\;on $U$,  conservation of mass then allows us to conclude \eqref{main000}.
This proves $\tilde{p}^+\in \Sigma(\rho_1)$ and thus $p_0^+\leq p_1^+$ in $\Omega$.

\medskip

Next we show $p_0^-\leq p_1^-$.
Suppose $b\mapsto \partial s^*(b,x)$ is strictly increasing for a.e.\;$x\in \Omega$.
Then by Lemma \ref{lem:dual_relation}, for a.e.\;$x\in \Omega$, $\partial s(z,x)$ only contains one element for all $z$ in $\partial s^*(\RR,x):=\cup_{b\in \RR}\partial s^*(b,x)$.
Since $\rho_0^*\leq \rho_1^*$ a.e.\;and $\partial s(\cdot,x)$ is increasing, we find $p_0^-(x)\leq p_1^-(x)$ a.e..
This further implies $p_0^-(x)\leq p_1^-(x)$ everywhere since $p_i^-$ are Lipschitz. (In fact, in this case, the maximizing pressure is unique.)

If $b\mapsto \partial s^*(b,x)$ is not strictly increasing for almost all $x$, we can apply the argument in Lemma \ref{lem:pressure_special} to construct a sequence of $s_k^*(b,x)$ that have $\partial s_k^*(b,x)$ strictly increasing, and two sequences of ordered maximizing pressures $p^*_{0,k}(x)\leq p_{1,k}^*(x)$ with $p_{i,k}^*$ converging uniformly to the smallest $c$-concave maximizer $p_i^-$.
%$$
%p_i^*\in\argmax_{p\in X^*,\, p^{c\bar{c}}=p} J^*(p,\rho_i).
%$$
As a result, $p_0^-\leq p_1^-$.
\end{proof}

We conclude this section with a lemma providing construction of a family of stationary densities that will serve as stationary barriers.
Later, we will see that they can give uniform bounds on the discrete densities and pressures over iterations.
%We will see later that it allows us to It will be used to prove Proposition \ref{prop:spatial_equicontinuity}.
We will assume \ref{assumption: convexity}-\ref{assumption: inf sup ratio of partial s*}.
Since in this case $s^*(\cdot,x)$ will be differentiable, instead of the subdifferential $\partial s^*(\cdot,x)$, we shall write the partial derivative of $s^*$ with respect to the first variable as $\partial_p s^*(\cdot,x)$.

\begin{lemma}\label{lem:lower_bound}
Suppose that $s$ satisfies \ref{assumption: convexity}-\ref{assumption: inf sup ratio of partial s*}.  For any $\lambda$ satisfying
\begin{equation}\label{eqn: condition on admissible total mass}
0<\lambda< \lim_{b\to\infty} \int_{\Omega} \partial_p s^*(b,x)\, dx,
\end{equation}
the variational problem
\begin{equation}\label{eq:fixed_mass_problem}
\inf_{\rho\in X\atop\int_{\Omega} \rho(x)\, dx=\lambda} E(\rho)
\end{equation}
admits a unique (in the a.e.\;sense) minimizer $\rho_\lambda$.
It has the following properties:
\begin{enumerate}%[label = (\alph*)]

\item \label{property: minimal pressure for static density}
There exists a minimal $\alpha_{\lambda}\in \RR$ such that $\rho_{\lambda}=\partial_p s^*(\alpha_{\lambda}, x)$ a.e.. By minimality, we mean that if some $\alpha\in\RR$ satisfies $\rho_\lambda = \partial_p s^*(\alpha,x)$ a.e., then we must have $\alpha\geq \alpha_\lambda$.

\item $\rho_{\lambda}$ is non-decreasing with respect to $\lambda$.

\item There exists $0<a_{\lambda}<b_{\lambda}<\infty$ such that $\rho_{\lambda}(x) \in [a_{\lambda}, b_{\lambda}]$ for almost all $x\in\Omega$.
\item\label{property: comparison with static densities}
Given $\bar{\rho}\in X$ satisfying \eqref{eq:good_data0}, let
\[
\rho^*:=\argmin_{\rho\in X} J(\rho, \bar{\rho}),
\]
and let
\[
p^* \in \argmax_{p\in X^*, \,p^{c\bar{c}}=p} J^*(p,\bar{\rho})
\]
be the smallest $c$-concave maximizing pressure as is constructed in Lemma \ref{lem:pressure_special}.
If for some $\lambda$ satisfying \eqref{eqn: condition on admissible total mass}, we have $\bar{\rho}\geq \rho_{\lambda}$ (resp.\;$\bar{\rho}\leq \rho_\lambda$) almost everywhere,
then $\rho^*\geq \rho_{\lambda}$ (resp.\;$\rho^*\leq \rho_\lambda$) almost everywhere and $p^*\geq \alpha_{\lambda}$ (resp.\;$p^*\leq \alpha_\lambda$).
%In particular,  is bounded from below by $\alpha_\lambda$.
\end{enumerate}
\end{lemma}
\begin{proof}
By the assumption \ref{assumption: regularity of s*}, $\alpha\mapsto \partial_p s^*(\alpha, x)$ is continuous and non-decreasing for all $x\in\Omega$.
Then \eqref{eqn: condition on admissible total mass} and \ref{assumption: asymptotics of s*} imply that there exists a non-empty bounded closed interval $[\alpha_\lambda, \beta_\lambda]\subset \RR$ such that
\begin{equation*}
\int_{\Omega} \partial_p s^*(\alpha, x)\, dx=\lambda
%\label{mass_constraint}
\end{equation*}
if and only if $\alpha\in[\alpha_\lambda, \beta_\lambda]$.
%Thus, $\rho_{\lambda}$ exists and $\int_{\Omega} \rho_{\lambda}(x)\, dx=\lambda$.

We claim that $\tilde{\rho}_{\lambda}:=\partial_p s^*(\alpha_{\lambda}, x)$ is a minimizer of \eqref{eq:fixed_mass_problem}.
Firstly, $\tilde{\rho}_{\lambda}\in X$.
Indeed, Lemma \ref{lem:dual_relation} gives
\[
\int_{\Omega} \alpha_{\lambda} \tilde{\rho}_{\lambda}(x)\, dx=\int_{\Omega} s\big(\tilde{\rho}_{\lambda}(x), x\big)+s^*(\alpha_{\lambda},x)\, dx.
\]
Recalling that $s^*(\alpha_{\lambda},x)\geq 0$, we have the bound
\[
E(\tilde{\rho}_{\lambda})=\int_{\Omega} s\big(\tilde{\rho}_{\lambda}(x), x\big)\, dx\leq \int_{\Omega} \alpha_{\lambda} \tilde{\rho}_{\lambda}(x)\, dx=\alpha_{\lambda}\lambda<+\infty.
\]
Now let $\rho$ be some other density with mass $\lambda$.  The convexity of the energy implies that
\[
E(\rho)\geq E(\tilde{\rho}_{\lambda})+\int_{\Omega}\alpha_{\lambda} \big(\rho(x)-\tilde{\rho}_{\lambda}(x))\, dx=E(\tilde{\rho}_{\lambda}).
\]
Hence, $\tilde{\rho}_{\lambda}$ is a minimizer of \eqref{eq:fixed_mass_problem}.

Conversely, for any minimizer $\bar{\rho}\in X$ of \eqref{eq:fixed_mass_problem}, let  % \textcolor{red}{(We need $c(y,y)\leq c(x,y)$ to claim this.)}
\[
\bar{\rho}_* =\argmin_{\rho\in X} J(\rho, \bar{\rho}).
\]
Then we should have $\bar{\rho}_* = \bar{\rho}$ a.e., and the transport map from $\bar{\rho}$ to $\bar{\rho}_*$ is $\bar{\rho}$-a.e.\;an identity map.
Let $\bar{p}$ be an arbitrary maximizer of the dual problem
\[
\bar{p} \in\argmax_{p\in X^*,p^{c\bar{c}}} J^*(p, \bar{\rho}).
\]
By Proposition \ref{prop:primal_dual}, $\bar{p}$ is Lipschitz
with $\nabla\bar{p} = 0$ for $\bar{\rho}$-a.e.\;$x\in\Omega$, and $\bar{\rho}(x) = \partial_p s^*(\bar{p}(x),x)$ a.e.\;in $\Omega$. %\textcolor{red}{(We need $\nabla_x c(y,y) = 0$ in order to derive the latter by virtue of \eqref{eq:forward_map}.)}
Let $Q = \{b\in\mathbb{R}:\essinf_{x\in \Omega} \partial_p s^*(b,x)>0\}$; clearly, $Q$ is an open set since $Q^c$ is closed by \ref{assumption: regularity of s*}.
Consider the set $V:=\bar{p}^{-1}(Q)$, which is also open in $\Omega$ as $\bar{p}(x)$ is continuous.
%The following argument is similar to the proof of Proposition \ref{prop:maximum_principle}.
Since $\bar{\rho}(x)=\partial_p s^*(\bar{p}(x),x)$ a.e., we find by the definition of $Q$ that $\bar{\rho}>0$ a.e.\;in $V$.
Hence, that $\nabla \bar{p}(x) = 0$ $\bar{\rho}$-a.e.\;in $V$ implies $\nabla \bar{p}(x) = 0$ a.e.\;in $V$.
So $\bar{p}$ is a constant on every connected component of $V$.
Since $\bar{p}$ is continuous, this implies $V$ is both open and closed in $\Omega$, and thus $V=\Omega$, which means $\bar{p}$ is a constant on $\Omega$.
Therefore, we conclude that any minimizer $\bar{\rho}$ must have the form $\bar{\rho}=\partial_p s^*(\alpha,x)$ in the a.e.\;sense for some $\alpha\in \RR$. Since $\alpha\mapsto \partial_p s^*(\alpha,x)$ is continuous and non-decreasing for all $x$, $\tilde{\rho}_{\lambda}$ turns out the be the unique minimizer of \eqref{eq:fixed_mass_problem}.
So from now on, we shall write it as $\rho_\lambda$.

The minimality of $\alpha_\lambda$ is then obvious given its definition and the argument above.
Since $\alpha_\lambda$ is clearly increasing in $\lambda$, $\rho_\lambda$ is non-decreasing with respect to $\lambda$.

By \ref{assumption: inf sup ratio of partial s*}, whenever $\lambda>0$, we must have
\[
a_\lambda:=\essinf_{x\in\Omega}\partial_p s^*(\alpha_\lambda,x)>0, \quad\mbox{ and }\quad b_\lambda:=\esssup_{x\in\Omega}\partial_p s^*(\alpha_\lambda,x)<+\infty,
\]
and vice versa.

Finally, the last claim follows from Theorem \ref{thm: direct comparison of rho and p}, the fact that
\[
\rho_{\lambda}=\argmin_{\rho\in X} J(\rho, \rho_{\lambda}),
\]
the minimality of $\alpha_\lambda$, as well as the argument above on the maximizers of the dual problem.
\end{proof}

\section{The Minimizing Movement Scheme and $L^1$-equicontinuity}
\label{sec: mms and equicontinuity}
In the rest of the paper, we aim at obtaining a weak solution of the problem $(P)$.
From now on, we assume $s(z,x)$ satisfies \ref{assumption: convexity}-\ref{assumption: inf sup ratio of partial s*} and we take $c(x,y)$ to be the quadratic cost
\begin{equation}\label{eq:mms_cost}
c(x,y)=\frac{1}{2\tau}|x-y|^2,
\end{equation}
where the parameter $\tau>0$ plays the role of the time step.
With $\rho^{0,\tau} :=\rho_0\in X$ satisfying \eqref{eq:good_data0}, thanks to Proposition \ref{prop:primal_dual}, we may apply the minimizing movements scheme
\begin{align}
\rho^{n+1,\tau}=&\;\argmin_{\rho\in X} J(\rho,\rho^{n,\tau}),\label{eqn: primal scheme recall}\\
p^{n+1,\tau}\in &\;\argmax_{p\in X^*, \,p^{c\bar{c}}=p} J^*(p,\rho^{n,\tau}),\label{eqn: dual scheme recall}
\end{align}
iteratively.
Since the maximizer of the dual problem may not be unique, $p^{n+1,\tau}$ here is always chosen to be the smallest $c$-concave one which is constructed in Lemma \ref{lem:pressure_special}.
Let $(\rho^\tau,p^\tau)$ be the time interpolation of the discrete solution defined in \eqref{interpolation}.
We hope to obtain a weak solution of $(P)$ in an appropriate sense by sending $\tau \to 0$.

In this section, we will focus on establishing compactness for the family of densities $\{\rho^{\tau}\}_{\tau>0}$.  In particular, the $L^1$-contraction principle enables us to prove a crucial $L^1$-spatial equicontinuity property for $\{\rho^{\tau}\}_{\tau>0}$.

We begin with establishing the following energy dissipation inequality, which is a well-known consequence of the minimizing movements scheme.

\begin{lemma}\label{lem: edi}
Let $\rho_0\in X$ satisfy \eqref{eq:good_data0}. Let $\rho^{\tau}$ and $p^{\tau}$ be given in \eqref{interpolation} with initial data $\rho_0$.
Then for any $T>0$, we have
\begin{equation}\label{eq:edi}
E(\rho^{\tau}(\cdot,T))+\frac{1}{2}\int_0^{T'} \int_{\Omega} \rho^{\tau}(x,t)|\nabla p^{\tau}(x,t)|^2\, dx\, dt\leq E(\rho_0).
\end{equation}
where $T':=(N_\tau+1)\tau$ and $N_\tau :=\lfloor \frac{T}{\tau}\rfloor$.
In particular, %with $M:= >-\infty$,\textcolor{blue}{why?}
\begin{equation}\label{pressure_bound}
\frac{1}{2}\int_0^{\infty} \int_{\Omega} \rho^{\tau}(x,t)|\nabla p^{\tau}(x,t)|^2\, dx\, dt\leq E(\rho_0)-\inf_{\rho\in X} E(\rho) <+\infty,
\end{equation}
where the bound only depends on $s$ and $E(\rho_0)$.
\end{lemma}
\begin{proof}
The optimality condition for the primal problem implies that
\begin{equation}\label{eq:primal_optimality}
J(\rho^{n+1,\tau},\rho^{n,\tau})\leq J(\rho^{n,\tau},\rho^{n,\tau}).
\end{equation}
For the quadratic cost \eqref{eq:mms_cost}, by \eqref{eq:quadratic_p_map}, the optimal transport map from $\rho^{n+1,\tau}$ to $\rho^{n,\tau}$ is given by $T_{p^{n+1,\tau}}^{-1}(x)=x+\tau\nabla p^{n+1,\tau}(x)$ in the $\rho^{n+1,\tau}$-a.e.-sense.  Hence, % the quadratic transportation cost between $\rho^{n,\tau}$ and $\rho^{n+1,\tau}$ is
\[
\frac{1}{2\tau}W_2^2(\rho^{n+1,\tau},\rho^{n,\tau}) = \int_{\Omega} \frac{1}{2\tau}|T_{p^{n+1,\tau}}^{-1}(x)-x|^2\rho^{n+1,\tau}(x)\, dx=\frac{\tau}{2}\int_{\Omega} |\nabla p^{n+1,\tau}(x)|^2\rho^{n+1,\tau}(x)\, dx.
\]
So \eqref{eq:primal_optimality} can be rewritten as
\[
\frac{\tau}{2}\int_{\Omega} |\nabla p^{n+1,\tau}(x)|^2\rho^{n+1,\tau}(x)\, dx\leq E(\rho^{n,\tau})-E(\rho^{n+1,\tau}).
\]
Summing over $n$ from $0$ to $N_\tau$, we have
\[
E(\rho^{N_\tau+1,\tau})+\frac{\tau}{2}\sum_{n=0}^{N_\tau} \int_{\Omega} |\nabla p^{n+1,\tau}(x)|^2\rho^{n+1,\tau}(x)\, dx\leq E(\rho_0).
\]
This together with the definition of $\rho^\tau$ and $p^\tau$ yields \eqref{eq:edi}.
\eqref{pressure_bound} follows from the fact that $\inf_{\rho\in X} E(\rho)>-\infty$ thanks to \ref{assumption: effective domain}.
\end{proof}

The energy dissipation inequality gives us a control on the gradients of the pressure, which provides enough regularity for establishing the aforementioned $L^1$-spatial equicontinuity of $\{\rho^\tau\}_{\tau}$.

\begin{prop}\label{prop:spatial_equicontinuity}
Let $s$ satisfy \ref{assumption: convexity}-\ref{assumption: inf sup ratio of partial s*}, and let $\rho_0\in X$ satisfy \eqref{eq:good_data0}.  For $\rho^{\tau}$ as given in \eqref{interpolation} with initial data $\rho_0$, extend $\rho^{\tau}$ by zero to all of $\RR^d$. Then for any $T>0$ and $y\in \R^d$ we have
\begin{equation*}%\label{prop1}
\lim_{\epsilon\to 0}\, \sup_{0<\tau\ll T}\, \int_0^T \int_{\Omega} |\rho^{\tau}(x+\epsilon y,t)-\rho^{\tau}(x,t)|\,dx\,dt=0.
\end{equation*}
\end{prop}

In order to prove Proposition \ref{prop:spatial_equicontinuity}, we need the following lemma that allows us to approximate $\partial_p s^*(p,x)$ by smooth functions.

\begin{lemma}
\label{lem: approximation by smooth functions}
If $f\in L^1(\Omega; C_{loc}(\RR))$ such that $p\mapsto f(p,x)$ is monotone for all $x\in\Omega$, then there exists a sequence of $f_m$ smooth on $\RR\times \Omega$ such that for any compact interval $I\subset \RR$,
\[
\lim_{m\to\infty} \int_{\Omega} \sup_{p\in I} |f(p,x)-f_m(p,x)|\, dx=0.
\]
\end{lemma}
\begin{proof}
Let us extend $f$ to $L^1(\RR^d; C_{loc}(\RR))$ by setting $f(p,x)=0$ if $x\notin \Omega$.
Let $\eta$ be a nonnegative smooth mollifier supported in the unit ball of $\RR\times \RR^d$, having integral $1$.
We define
\[
f_m(p,x):=\int_{\RR\times \RR^d} \eta(q, y) f\left(p+\frac{1}{m}q, x+\frac{1}{m}y\right)\, dy\, dq.
\]
It then follows that
 \[
\int_{\Omega} \sup_{p\in I} |f(p,x)-f_m(p,x)|\, dx\leq \int_{\Omega}\int_{\RR\times \RR^d} \eta(q,y) \sup_{p\in I} \left|f(p,x)-f\left(p+\frac{1}{m}q,x+\frac{1}{m}y\right)\right|.
\]
With $k\in\ZZ_+$ to be chosen, we subdivide $I$ into $k$ disjoint intervals of equal length, say $I=\bigcup_{0\leq i\leq k-1} [a_{i,k},a_{i+1,k}]$ with $a_{0,k}\leq\cdots\leq a_{k,k}$.
Thanks to the monotonicity of $p\mapsto f(p,x)$, for any fixed $x\in \Omega$, $|y|\leq 1$, $|q|\leq 1$, and $k\leq m$, we have
\begin{equation*}
\begin{split}
&\;\sup_{p\in [a_{i,k}, a_{i+1,k}]} \left|f(p,x)-f\left(p+\frac{1}{m}q,x+\frac{1}{m}y\right)\right|\\
\leq&\;
\left|f(a_{i+1,k},x)-f\left(a_{i,k}-\frac{1}{k},x+\frac{1}{m}y\right)\right|+ \left|f(a_{i,k},x)-f\left(a_{i+1,k}+\frac{1}{k},x+\frac{1}{m}y\right)\right|.
\end{split}
\end{equation*}
Hence, for each $k\in \ZZ_+$, we have
\begin{equation}
\begin{split}
&\;\int_{\Omega}\int_{\RR\times \RR^d} \eta(q,y) \sup_{z\in I} \left|f(p,x)-f\left(p+\frac{1}{m}q,x+\frac{1}{m}y\right)\right|\\
\leq&\;
 \int_{\Omega}\int_{\RR\times \RR^d} \eta(q,y) \max_{0\leq i\leq k-1} \left|f(a_{i+1,k},x)-f\left(a_{i,k}-\frac{1}{k},x+\frac{1}{m}y\right)\right|\\
 &\;+  \int_{\Omega}\int_{\RR\times \RR^d} \eta(q,y) \max_{0\leq i\leq k-1}\left|f(a_{i,k},x)-f\left(a_{i+1,k}+\frac{1}{k},x+\frac{1}{m}y\right)\right|.
\end{split}
\label{eqn: difference between mollified and original f}
\end{equation}

We can then bound the first term above by splitting it into two terms
\begin{equation}
\label{eq:max_sum_bound}
\begin{split}
&\;\int_{\Omega}\int_{\RR\times \RR^d} \eta(q,y) %\max_{0\leq i\leq k-1} \left|f(a_{i,k},x)-f\left(a_{i+1,k}+\frac{1}{k},x\right)\right|+
\max_{0\leq i\leq k-1} \left|f(a_{i+1,k},x)-f\left(a_{i,k}-\frac{1}{k},x\right)\right|\\
&\;+\sum_{i=0}^{k-1} \int_{\RR^d}\int_{\RR\times \RR^d} \eta(q,y) %\left|f\left(a_{i+1,k}+\frac{1}{k},x\right) -f\left(a_{i+1,k}+\frac{1}{k},x+\frac{1}{m}y\right)\right|+
\left|f\left(a_{i,k}-\frac{1}{k},x\right)-f\left(a_{i,k}-\frac{1}{k},x+\frac{1}{m}y\right)\right|.
\end{split}
\end{equation}
By assumption, for any fixed $z\in \RR$, $x\mapsto f(z,x)$ is an $L^1(\RR^d)$ function. So
\begin{equation*}
\begin{split}
\lim_{m\to\infty}\sum_{i=0}^{k-1} \int_{\RR^d}\int_{\RR\times \RR^d} \eta(q,y) %|f(a_{i,k+1}+\frac{1}{k},x)-f(a_{i+1,k}+\frac{1}{k},x+\frac{1}{m}y)|+
\left|f\left(a_{i,k}-\frac{1}{k},x\right)-f\left(a_{i,k}-\frac{1}{k},x+\frac{1}{m}y\right)\right|=0
\end{split}
\end{equation*}
for any finite $k$.  We also know that for almost every $x\in\Omega$, by uniform continuity of $f(\cdot,x)$ on bounded intervals, % we have
\[
\lim_{k\to\infty} %\max_{0\leq i\leq k-1} |f(a_{i,k},x)-f(a_{i+1,k}+\frac{1}{k},x)|+
\max_{0\leq i\leq k-1} \left|f(a_{i+1,k},x)-f\left(a_{i,k}-\frac{1}{k},x\right)\right|=0.
\]
Therefore, the first line in \eqref{eq:max_sum_bound} will vanish as $k\to\infty$ thanks to the dominated convergence theorem.
By letting $m\geq k\gg 1$, we can make the first term in \eqref{eqn: difference between mollified and original f} as small as we want.
The second term in \eqref{eqn: difference between mollified and original f} can be handled in exactly the same way.
This completes the proof.
\end{proof}

Now we are ready to prove the $L^1$-spatial equicontinuity of $\{\rho^\tau\}_{\tau}$.

\begin{proof}[Proof of Proposition \ref{prop:spatial_equicontinuity}]

We first approximate the initial data by densities that are bounded away from zero and infinity
$$
\rho_{0,k}(x): =\min\big(\max\big( \rho_0(x), \rho_{\frac{1}{k}}(x)\big), \partial_p s^*(k,x)\big),
$$
where $\rho_{\frac{1}{k}}$ is defined as in Lemma \ref{lem:lower_bound}.
By \ref{assumption: asymptotics of s*}, for $k\gg 1$,
%begin{equation}\label{eqn: range of cutoff densities}
$\rho_{0,k}(x)\in [\rho_{\frac{1}{k}}(x), \partial_p s^*(k,x)]$
%end{equation}
is well-defined.
It is clear that $\|\rho_{0,k}-\rho_0\|_{L^1(\Omega)}\to 0$ as $k\to \infty$.
Moreover, by the convexity of $s$, $\rho_{0,k}\in X$ and it obviously satisfies \eqref{eq:good_data0} for fixed $k$.
Let $\rho_k^{n,\tau}$ $(n\geq 1)$ be as given in \eqref{eqn: primal scheme recall} with initial data $\rho_k^{0,\tau} := \rho_{0,k}$.
%Since $\partial s(\rho_{0,k},x)\in[\zeta_k, k]$, by Lemma \ref{lem:comparison_for_pressure},
%Also by Lemma~\ref{bounds} we have $p^{n,\tau}_k \in [-A_k,A_k]$ a.e.\;in $\Omega$, where $A_k$ can be chosen independent of $n$ and $\tau$.
We extend $\rho_k^{n,\tau}$ by zero to the entire $\mathbb{R}^d$.
Let $p_k^{n,\tau}$ $(n\geq 1)$ be obtained iteratively by \eqref{eqn: dual scheme recall} with $\rho^{n,\tau}$ there replaced by $\rho_k^{n,\tau}$.
Here we take $p_k^{n,\tau}$ to be the smallest $c$-concave maximizer of \eqref{eqn: dual scheme recall} in all steps.
Thanks to Lemma \ref{lem:lower_bound}, for all $n$ and $\tau$, we have $\rho_k^{n,\tau}\geq \rho_{\frac{1}{k}}\geq a_{\frac{1}{k}}>0$ a.e., and $p_k^{n,\tau}\in[\alpha_{\frac{1}{k}},k]$.
Let $\rho_k^{\tau}$ and $p_k^\tau$ be the discontinuous time interpolations as given in \eqref{interpolation}.

Without loss of generality, assume $y\in B_1(0)\subset \mathbb{R}^d$. We have
\begin{equation*}
\begin{split}
&\;\sup_{\tau} \int_{\Omega} |\rho^{n,\tau}(x+\epsilon y)-\rho^{n,\tau}(x)|\,dx\\
%\leq &\;\int_{\Omega} |\rho_k^{n,\tau}(x+\epsilon y)-\rho^{n,\tau}(x+\epsilon y)|+ |\rho_k^{n,\tau}(x)-\rho^{n,\tau}(x)|\,dx\\
%&\;+\int_{\Omega} |\rho^{n,\tau}_k(x+\epsilon y)-\rho^{n,\tau}_k(x)|\,dx\\
\leq &\;2\sup_{\tau} \|\rho_k^{n,\tau}-\rho^{n,\tau}\|_{L^1(\Omega)}+\sup_{\tau} \int_{\Omega} |\rho^{n,\tau}_k(x+\epsilon y)-\rho^{n,\tau}_k(x)|\,dx.
\end{split}
%\label{eqn: approximate rho by rho_k}
\end{equation*}
By Theorem~\ref{thm:l1_contraction}, $\sup_{\tau} \|\rho_k^{n,\tau}-\rho^{n,\tau}\|_{L^1(\Omega)}\leq \|\rho_{0,k}-\rho_0\|_{L^1(\Omega)}$, which converges to $0$ as $k\to \infty$. Hence to conclude, it is enough to show that for every fixed $k$,
\begin{equation}\label{claim0}
\lim_{\e\to 0 } \sup_{\tau} \int_0^T \int_{\Omega} |\rho^{\tau}_k(x+\epsilon y,t)-\rho^{\tau}_k(x,t)|\,dx\,dt = 0.
\end{equation}

Define $\Omega_{\epsilon_0} :=\{x\in \Omega: \; B_{\epsilon_0}(x)\subset \Omega\}$. Then for any $\e\in (0,\e_0)$,
\begin{equation*}
\begin{split}
&\;\int_{\Omega_{\epsilon_0}}a_{\frac{1}{k}} |p_{k}^{n,\tau}(x+\epsilon y)-p_{k}^{n,\tau}(x)|\,dx\\
\leq &\;|\epsilon y|\int_{\Omega_{\epsilon_0}}a_{\frac{1}{k}}\int_0^1 |\nabla p_{k}^{n,\tau}(x+\theta\epsilon y)|\,d\theta \,dx \\
\leq &\;|\epsilon y|\int_0^1 \big(a_{\frac{1}{k}} |\Omega_{\epsilon_0}|\big)^{1/2}\left(\int_{\Omega_{\epsilon_0}}\rho_{k}^{n,\tau}(x+\theta\epsilon y) |\nabla p_{k}^{n,\tau}(x+\theta\epsilon y)|^2\,dx\right)^{1/2} \,d\theta \\
\leq &\;\epsilon \big(a_{\frac{1}{k}} |\Omega|\big)^{1/2}\left(\int_{\Omega}\rho_{k}^{n,\tau}(x) |\nabla p_{k}^{n,\tau}(x)|^2\,dx\right)^{1/2}.
\end{split}
\end{equation*}
Note that $p_{k}^{n,\tau}$ is unique on $\Omega$ up to an additive constant, so the quantity in the first line is well-defined.
Combining this with Lemma \ref{lem: edi} yields
\begin{equation}
%\begin{split}
\tau\sum_{n = 1}^{N_\tau+1} \int_{\Omega_{\epsilon_0}}|p_{k}^{n,\tau}(x+\epsilon y)-p_{k}^{n,\tau}(x)|\,dx%\\
\leq C(k,|\Omega|, T,s,E(\rho_0))\epsilon.
\label{eqn: L^1 bound on displacement of p}
\end{equation}
Here given $\tau\ll T$, we assumed $\tau (N_\tau+1) \leq 2T$.
Recall that from Proposition \ref{prop:primal_dual} we have $\rho_k^{n,\tau} =\partial_p s^*(p^{n,\tau}_k,x)$ a.e.. Next, we will use this relation as well as \eqref{eqn: L^1 bound on displacement of p} to conclude.

\medskip

Thanks to Lemma \ref{lem: approximation by smooth functions}, we take $D_{\delta,k}(p,x)$ to be a smooth approximation of $\partial_p s^*(p,x)$ on $\Sigma_k:= [\alpha_{\frac{1}{k}}, k]\times \Omega$, such that
\begin{equation*}%\label{approx_smooth}
\int_{\Omega}\|\partial_p s^*(\cdot,x)-D_{\delta,k}(\cdot,x)\|_{C([\alpha_{1/k},k])}\,dx\leq \delta.
\end{equation*}
With $\e\in (0,\e_0)$ and $M_{\delta,k} := \sup_{\Sigma_k}|\nabla D_{\delta,k}|$, we derive that
\begin{equation*}
\begin{split}
&\;\int_{\Omega_{\epsilon_0}} |\rho_k^{n,\tau}(x+\epsilon y)- \rho_k^{n,\tau}(x)|\,dx\\
= &\;\int_{\Omega_{\epsilon_0}} |\partial_p s^*(p_k^{n,\tau}(x+\epsilon y),x+\epsilon y)- \partial_p s^*(p_k^{n,\tau}(x),x)|\,dx\\
%
%\leq &\;\int_{\Omega_{\epsilon_0/2}} |\partial_p s^*(p_k^{n,\tau}(x),x)- \partial_p s^*(p_k^{n,\tau}(x-\epsilon y),x)|\,dx \\
%&\;+\int_{\Omega_{\epsilon_0}} \sup_{p\in [\zeta_k, k]}|\partial_p s^*(p,x+\epsilon y)- \partial_p s^*(p,x)|\,dx.
%
\leq &\;\int_{\Omega_{\epsilon_0}} |D_{\delta,k}(p_k^{n,\tau}(x+\e y),x+\e y)- D_{\delta,k}(p_k^{n,\tau}(x),x)|\,dx+C\delta\\
\leq & \; M_{\delta, k} \int_{\Omega_{\epsilon_0}} |p_k^{n,\tau}(x+\e y) - p_k^{n,\tau}(x)|+|\epsilon y|\, dx + C\delta,
\end{split}
%\label{eqn: splitting the L^1 difference between rho and a displaced rho}
\end{equation*}
where in the first inequality, we used the fact that $x+\e y\in \Omega$ for all $x\in \Omega_{\e_0}$.  Due to \eqref{eqn: L^1 bound on displacement of p}, we conclude that
$$
\sup_{\tau} \int_0^T\int_{\Omega_{\epsilon_0}} |\rho_k^{\tau}(x+\epsilon y,t)- \rho_k^{\tau}(x,t)|\,dx \,dt  \leq M_{\delta,k} C(k,|\Omega|, T,s,E(\rho_0))\epsilon + CT\delta.
$$
%Combining this with \eqref{eqn: approximate rho by rho_k}, we derive with $\e\in (0,\e_0)$ that

To this end, we derive that
\begin{equation*}
\begin{split}
&\;\sup_{\tau}\int_0^T\int_{\Omega} |\rho_k^{\tau}(x+\epsilon y,t)- \rho_k^{\tau}(x,t)|\,dx\, dt \\
\leq &\;\sup_{\tau}\int_0^T\int_{\Omega_{\e_0}} |\rho_k^{\tau}(x+\epsilon y,t)- \rho_k^{\tau}(x,t)|\,dx\,dt
+C\sup_{\tau}\int_0^T\int_{\Omega\setminus \Omega_{2\e_0}} |\rho_k^{\tau}(x,t)|\,dx\,dt\\
\leq &\;M_{\delta,k} C(k,|\Omega|, T,s,E(\rho_0))\epsilon +CT\delta+CT\int_{\Omega\setminus \Omega_{2\e_0}} |\partial_p s^{*}(k,x)|\,dx.
\end{split}
\end{equation*}
Now sending $\e\to 0$, we find % we conclude that
$$
\limsup_{\e\to 0}\sup_{\tau}\,\int_0^T\int_{\Omega} |\rho_k^{\tau}(x+\epsilon y,t)- \rho_k^{\tau}(x,t)|\,dx\,dt  \leq CT\delta+CT\int_{\Omega\setminus \Omega_{2\e_0}} |\partial_p s^{*}(k,x)|\,dx.
$$
Therefore, we can conclude \eqref{claim0}  by sending $\delta$ and $\e_0\to 0$.
\end{proof}

In the rest of this paper, we shall assume that $\rho_0\in X$ satisfies \eqref{good_data_1}, which is a stronger assumption than \eqref{eq:good_data0}. As we will see below, this gives rise to uniform upper bounds for both the density and pressure variables over the iteration.

\begin{lemma}\label{bounds}
Let $\rho_0\in X$ satisfy \eqref{good_data_1}. Then for all $n$ and $\tau$, $\rho^{n,\tau}$ satisfies \eqref{good_data_1},
\begin{equation*}
\norm{\rho^{n,\tau}}_{L^{\infty}(\Omega)}\leq C(s,M)\quad \mbox{ and }\quad\sup_{x\in \Omega}p^{n,\tau}(x)\leq M,
\end{equation*}
where $M$ is introduced in \eqref{good_data_1}.
\end{lemma}
\begin{proof}
Let
$$
\lambda = \int_\Omega \partial_p s^*(M,x)\,dx.
$$
By \eqref{good_data_1}, it satisfies \eqref{eqn: condition on admissible total mass}.
By Lemma \ref{lem:lower_bound}\ref{property: minimal pressure for static density}, $\rho_\lambda = \partial_p s^*(M,x)$ a.e..
Using Lemma \ref{lem:lower_bound}\ref{property: comparison with static densities} iteratively, for all $n$ and $\tau$, we have that $\rho^{n,\tau}\leq \rho_\lambda$ a.e., so \eqref{good_data_1} is carried over.
Beside, the smallest $c$-concave maximizer $p^{n,\tau}$ satisfies $p^{n,\tau}\leq \alpha_\lambda$.
Here $\alpha_\lambda$ is introduced in Lemma \ref{lem:lower_bound}.
Then the density bound follows from Lemma \ref{lem:pressure_upper_bound}, and the pressure bound follows from the fact $\alpha_\lambda\leq M$ due to the minimality of $\alpha_\lambda$.
\end{proof}

Now that we have a bound on the pressure gradient, we can bound discrete time derivatives of the density in $L^2([0,T]; H^{-1}(\R^d))$.

\begin{lemma} \label{lem:rho_time_bound}
Let $\rho_0\in X$ satisfies \eqref{good_data_1}. Let $\rho^\tau$ be extended by zero to the entire $\R^d$ and define $\rho^\tau(x,t):=\rho_0(x)$ for $t<0$.
Define $\sigma_{-\tau}\rho^\tau(x,t) := \rho^\tau(x,t-\tau)$.
Then
\begin{equation*}%\label{prop2}
\int_0^{\infty} \tau^{-2}\norm{\rho^{\tau}-\sigma_{-\tau}\rho^{\tau}}_{H^{-1}(\R^d)}^2\,dt\leq C(s,E(\rho_0),M),
\end{equation*}
where $M$ is given in \eqref{good_data_1}.
\end{lemma}
\begin{proof}
With $f\in C_0^\infty(\R^d)$,
\[
\int_{\Omega} \frac{\rho^{n+1,\tau}(x)-\rho^{n,\tau}(x)}{\tau} f(x)\, dx=\int_{\Omega} \frac{f(x)-f\big(x+\tau\nabla p^{n+1,\tau}(x)\big)}{\tau}\rho^{n+1,\tau}(x)\, dx.
\]
Applying the fundamental theorem of calculus and Cauchy-Schwarz inequality, we derive that %the previous line is equal to
\[
\begin{split}
\int_{\Omega} \frac{\rho^{n+1,\tau}(x)-\rho^{n,\tau}(x)}{\tau} f(x)\, dx
=&\;\int_{\Omega}\int_0^1  \nabla f\big(x+\tau \theta\nabla p^{n+1,\tau}(x)\big)\cdot \nabla p^{n+1,\tau}(x)\rho^{n+1,\tau}(x)\, d\theta\, dx\\
\leq &\;\norm{\nabla f}_{L^{2}(\tilde{\rho}^{n+1,\tau})}\norm{\nabla p^{n+1,\tau}}_{L^2(\rho^{n+1,\tau})},
\end{split}
\]
where
$\tilde{\rho}^{n+1,\tau}:=\int_0^1 \rho^{n+1,\tau}_{\theta}d\theta$ and $\rho_{\theta}^{n+1,\tau}$ is the displacement interpolant $(\id +\tau\theta\nabla p^{n+1,\tau})_{\#}\rho^{n+1,\tau}$.
Since $L^p$-norms are displacement convex \cite{OTAM}, we have by Lemma \ref{bounds} that
\[
\norm{\tilde{\rho}^{n+1,\tau}}_{L^{\infty}(\Omega)}\leq \max\big(\norm{\rho^{n,\tau}}_{L^{\infty}(\Omega)}, \norm{\rho^{n+1,\tau}}_{L^{\infty}(\Omega)}\big)\leq C(s,M).
\]
Hence,
\[
\tau^{-1}\norm{\rho^{\tau}-\sigma_{-\tau}\rho^{\tau}}_{H^{-1}(\R^d)}\leq C(s,M)\norm{\nabla p^{\tau}}_{L^2(\rho^{\tau})}.
\]
Taking square integral in time and using Lemma \ref{lem: edi} yields the desired estimate.
\end{proof}

\begin{prop}\label{prop:strong_compactness}
Suppose $\rho_0\in X$ satisfying \eqref{good_data_1} and let $\rho^\tau$ be as given in \eqref{interpolation} with initial data $\rho_0$.
%Suppose $\rho^\tau \rightharpoonup \rho$ in $L^1(\Omega_T)$.
Then there exists $\rho\in L^\infty(\Omega_T)$, where $\Omega_T :=\Omega\times [0,T]$, such that along a subsequence, $\rho^{\tau}\to \rho$ in $L^q(\Omega_T)$ as $\tau\to 0$ for all $q\in[1,\infty)$.
\end{prop}
\begin{remark}
Given Proposition \ref{prop:spatial_equicontinuity} and Lemma \ref{lem:rho_time_bound}, the proof is an adaptation of that of the Kolmogorov-M.\;Riesz-Fr\'{e}chet Theorem \cite[\S 4.5]{brezis2010functional}.
\end{remark}
\begin{proof}
Thanks to Lemma \ref{bounds}, $\{\rho^\tau\}_\tau$ is uniformly bounded in $L^\infty(\Omega_T)$.
So there exists $\rho\in L^\infty(\Omega_T)$, such that along a subsequence, $\rho^\tau \rightharpoonup \rho$ in $L^q(\Omega_T)$ as $\tau \to 0$ for all $q\in[1,\infty)$.
It then suffices to show the strong convergence for $q = 1$, as the other cases follow from the uniform boundedness of $\{\rho^\tau\}_\tau$ in $L^\infty(\Omega_T)$ and interpolation.

Let $\varphi\in C_0^\infty(\mathbb{R}^d)$ be a non-negative, radially symmetric mollifier in space, supported on $\overline{B_1(0)}$ and having integral $1$.
Let $\eta\in C_0^\infty(\mathbb{R})$ be a non-negative mollifier in time, supported on $[0,1]$ and having integral $1$.
Denote $\varphi_\e(x):= \e^{-d}\varphi(x/\e)$ and $\eta_\delta(x):= \delta^{-1}\eta(x/\delta)$.
Then for any $\e,\delta>0$, $\{\eta_\delta*\varphi_\e*\rho^\tau \}_\tau$ is uniformly bounded and they pointwise converge to $\eta_\delta*\varphi_\e*\rho$ in the space-time as $\tau \to 0$ along the subsequence.
%In what follows, we shall abbreviate $\Omega_T:= [0,T]\times \Omega$.
Then the dominated convergence theorem implies $\eta_\delta*\varphi_\e*\rho^\tau \to \eta_\delta*\varphi_\e*\rho$ in $L^1(\Omega_T)$.

Then we deduce that for arbitrary $0<\tau\ll T$,
\begin{equation*}
\begin{split}
\|\rho^{\tau} -\rho\|_{L^1(\Omega_T)}
\leq &\; \|\rho^{\tau} -\varphi_\e*\rho^{\tau}\|_{L^1(\Omega_T)} +\|\varphi_\e*\rho^{\tau} -\eta_\delta*\varphi_\e*\rho^{\tau}\|_{L^1(\Omega_T)}\\
&\;+\|\eta_\delta*\varphi_\e*\rho^{\tau} -\eta_\delta*\varphi_\e*\rho\|_{L^1(\Omega_T)}+\| \rho-\eta_\delta*\varphi_\e*\rho\|_{L^1(\Omega_T)}.
\end{split}
\end{equation*}
Note that $\|\rho -\eta_\delta*\varphi_\e*\rho\|_{L^1(\Omega_T)} \to 0$ as $\e,\delta \to 0$, since $\rho\in L^1(\Omega_T)$. Hence
 in order to prove $\|\rho^\tau -\rho\|_{L^1(\Omega_T)}\to 0$ as $\tau \to 0$, it suffices to show that
\begin{enumerate}
  \item \label{task: the epsilon limit space}$\|\rho^\tau -\varphi_\e*\rho^\tau\|_{L^1(\Omega_T)}\to 0$ as $\e\to 0$ uniformly in $\tau$;
  \item \label{task: the delta limit time}for any fixed $\e>0$, $\|\varphi_\e*\rho^\tau -\eta_\delta*\varphi_\e*\rho^\tau\|_{L^1(\Omega_T)}\to 0$ as $\delta \to 0$ uniformly in $\tau$.
\end{enumerate}

The first convergence can be justified by Proposition \ref{prop:spatial_equicontinuity}.
Indeed,
\begin{equation*}
\sup_{0<\tau\ll T}\|\rho^\tau-\varphi_\e*\rho^\tau\|_{L^1(\Omega_T)}\leq \int_{B_1(1)} \varphi (y)\sup_{0<\tau\ll T}\|\rho^\tau(x, t)-\rho^\tau(x - \e y,t)\|_{L^1(\Omega_T)}\,dy.
\end{equation*}
Then the dominated convergence theorem applies.
To show the second one, we note that for fixed $\e>0$, $\tau \ll T$, and $\psi\in L^{\infty}(\R^d)$,
\begin{equation*}
\begin{split}
\left|\int_{\R^d} \varphi_\e*(\rho^\tau - \sigma_{-\tau}\rho^\tau)\cdot \psi\,dx\right|
= &\;\left|\int_{\R^d} (\rho^\tau - \sigma_{-\tau}\rho^\tau)\cdot \varphi_\e*\psi\,dx\right|\\
\leq&\; C(\e)\|\rho^\tau - \sigma_{-\tau}\rho^\tau\|_{H^{-1}(\R^d)}\|\psi\|_{L^\infty(\R^d)},
\end{split}
\end{equation*}
which implies
\begin{equation*}%\label{eqn:L^1 time regularity of molified densities}
\|\varphi_\e*(\rho^\tau - \sigma_{-\tau}\rho^\tau)\|_{L^1(\R^d)}\leq C(\e)\|\rho^\tau - \sigma_{-\tau}\rho^\tau\|_{H^{-1}(\R^d)}.
\end{equation*}
Now we derive that
\begin{equation*}
\begin{split}
&\;\|\varphi_\e*\rho^\tau -\eta_\delta*\varphi_\e*\rho^\tau\|_{L^1(\Omega_T)}\\
\leq &\;
\left\| \int_{0}^\delta\eta_\delta(s)\|\varphi_\e*\rho^\tau (x,t) -\varphi_\e*\rho^\tau(x,t-s)\|_{L^1(\R^d)}\,ds \right\|_{L^1([0,T])}.
\end{split}
\end{equation*}
Assume $\delta \ll \tau$.
The integrand here is non-zero only when $t$ and $(t-s)$ do not belong to the same small interval of the form $[n\tau,(n+1)\tau)$.
So $t$ must lie in the right $\delta$-neighborhood of at least one of $\{0,\tau,\cdots, N_\tau\tau\}$.
Hence, with $T'$ defined in Lemma \ref{lem: edi},
\begin{equation*}
\begin{split}
&\;\|\varphi_\e*\rho^\tau -\eta_\delta*\varphi_\e*\rho^\tau\|_{L^1(\Omega_T)}\\
\leq &\; \sum_{n = 0}^{N_\tau}\int_{n\tau}^{n\tau+\delta} \int_{0}^\delta\eta_\delta(s)\|\varphi_\e*\rho^\tau (x,t) -\varphi_\e*\rho^\tau(x,t-s)\|_{L^1(\R^d)}\,ds \,dt\\
\leq &\; \sum_{n = 0}^{N_\tau}C\delta \|\varphi_\e*\rho^{n,\tau} -\varphi_\e*\rho^{n+1,\tau}\|_{L^1(\R^d)}\\
\leq &\; C(\epsilon)\delta\int_{0}^{T'}\tau^{-1}\|\rho^{\tau} -\sigma_{-\tau}\rho^{\tau}\|_{H^{-1}(\R^d)}\,dt\\
\leq &\; C(\e,T,s,E(\rho_0),M)\delta.
\end{split}
\end{equation*}
We used Lemma \ref{lem:rho_time_bound} in the last inequality.
Thus the second convergence follows.
\end{proof}

\section{Convergence to the Continuum Limit}
\label{sec: continumm limit}

In this section, we show that our discrete approximation $(\rho^{\tau}, p^{\tau})$ yields a weak solution of $(P)$ in the continuum limit $\tau\to 0$.
We first address the case of strictly positive initial data.

\begin{theorem}\label{thm:pde_existence1}
Let $\rho_0\in X$ satisfies \eqref{good_data_1}.
In addition, suppose $\rho_0\geq \rho_\lambda$ for some $\lambda>0$, where $\rho_\lambda$ is defined in Lemma \ref{lem:lower_bound}.
Fix $T>0$.
There exist $\rho\in L^{\infty}(\Omega_T)$ and $p\in L^2\big([0,T];H^1(\Omega)\big)\cap L^{\infty}(\Omega_T) $ being a weak solution of $(P)$, such that the following holds.
\begin{enumerate}[label=(\alph*)]
\item\label{L^1 convergence positive case}
$\rho^{\tau} \to \rho$ in $L^1(\Omega_T)$ along a subsequence;
\item\label{convergence of p tau to p with positive density} along a subsequence, $p^{\tau}\rightharpoonup p$ weakly in $L^2([0,T]; H^1(\Omega))$ and weak-$*$ in $L^\infty(\Omega_T)$, and $\rho^{\tau}\nabla p^{\tau}\rightharpoonup \rho\nabla p$ in $L^2(\Omega_T)$; moreover, $\rho |\nabla p|^2\leq \liminf_{\tau \to 0}\rho^\tau|\nabla p^\tau|^2$ and $p\leq M$ in $\Omega_T$;
\item \label{relation between p and rho positive case}
$p(x,t)\in \partial s(\rho(x,t),x)$ a.e.\;and $\rho(x,t) = \partial_p s^*(p(x,t),x)$ a.e.\;in $\Omega_T$;
and
\item \label{weak formulation positive case} For a.e.\;$t_0\in[0,T]$,
\begin{equation}\label{weak}
\int_0^{t_0}\int_{\Omega} \rho \partial_t \phi-\rho\nabla p\cdot \nabla \phi\, dx\, dt=\int_{\Omega} \rho(x,t_0)\phi(x,t_0)-\rho_0(x)\phi(0,x)\,dx
\end{equation}
any $\phi \in C^{\infty}(\Omega_T)$.
\end{enumerate}
\end{theorem}

For general initial data $\rho_0\in X$ satisfying \eqref{good_data_1}, continuum solutions of $(P)$ of a weaker sense are obtained, which are similar to those given by \cite{Carrillo}.

\begin{theorem}\label{thm:pde_existence2}
Let $\rho_0\in X$ satisfies \eqref{good_data_1}.
Fix $T>0$.
There exists $\rho\in L^{\infty}(\Omega_T)$ and a measurable $p$ with $p\leq M$ in $\Omega_T$, such that
\begin{enumerate}[label=(\alph*)]
 \item \label{L^1 convergence to continum limit}$\rho^{\tau} \to \rho$ in $L^1(\Omega_T)$ along a subsequence;
 \item\label{relation between p and rho} $p(x,t)\in \partial s(\rho(x,t),x)$ a.e.\;and $\rho(x,t) = \partial_p s^*(p(x,t),x)$ a.e.\;in $\Omega_T$;
 \item\label{weighted_pressure} With $m(x,t) := \nabla [s^*(p(x,t),x)] - \partial_x s^*(p(x,t),x)\in L^2(\Omega_T)$, we have for a.e.\;$t_0\in [0,T]$,
\begin{equation}\label{weak_eq}
\int_0^{t_0}\int_{\Omega} \rho\partial_t \phi-m\cdot \nabla \phi\, dx\, dt=\int_{\Omega} \rho(x,t_0)\phi(x,t_0)-\rho_0(x)\phi(x,0)\,dx
\end{equation}
 for any $\phi\in C^{\infty}(\Omega_T)$.
 \item\label{relation between m and grad p} $\frac{m}{\rho} = \nabla p$ in the support of $\rho$, in the sense that
 \begin{equation}\label{int_parts}
\int_0^T \int_{\Omega} \frac{m}{\rho} \cdot f = -\int_0^T\int_{\Omega} p \nabla\cdot f\, dx
 \end{equation}
for any $f\in \mathcal{T}$, where
\[
\mathcal{T}:=\Big\{f\in L^2([0,T]; H^1(\Omega)): f\cdot n |_{\partial \Omega \times [0,T]} = 0,\; \|\rho^{-1}f\|_{L^2(\Omega_T)}+\|\rho^{-1}\nabla\cdot f\|_{L^1(\Omega)}<+\infty\Big\}.
\]
\end{enumerate}
\end{theorem}

\begin{remark}
\begin{enumerate}%[label=(\arabic*)]
\item Our formulation of the weak solutions in Theorem~\ref{thm:pde_existence2} by \ref{weighted_pressure} coincides with the one in \cite{Carrillo}. We give an additional description of the transport velocity $-\nabla p$ by \ref{relation between m and grad p}, albeit with limited class of test functions.
\item   The continuum pressure $p$ in Theorem \ref{thm:pde_existence2} is obtained by an approximation argument using the pressure corresponding to a strictly positive initial density that is obtained in Theorem~\ref{thm:pde_existence1}.
    It is not clear whether it is always possible to recover $p$ from the discrete solution $p^{\tau}$.

\item  The set $\mathcal{T}$ of test functions in \ref{relation between m and grad p} is non-empty.
    See the proof of Theorem \ref{thm:pde_existence2}.
\end{enumerate}
\end{remark}

To prove these theorems, let us first show that the discrete solutions approximately satisfy the weak form of the continuum PDE $(P)$.

\begin{lemma}\label{continuity}
Fix $T>0$.
$(\rho^{\tau}, p^{\tau})$ satisfies
\begin{equation}
\int_{0}^{t_0} \int_{\Omega} \rho^{\tau} \partial_t\phi -\rho^{\tau}\nabla p^{\tau}\cdot \nabla \phi \, dx\,dt= \int_{\Omega} \rho^{\tau}(x,t_0)\phi(x,t_0) - \rho_0(x) \phi(x,0) \,dx +\epsilon_\tau \|\phi\|_{C^2(\Omega_T)},
\label{eqn: discrete weak formula almost equality}
\end{equation}
for all $\phi\in C_0^2(\Omega_T)$ and $t_0\in [2\tau, T]$.
Here $\epsilon_\tau$ satisfies $|\epsilon_\tau|\leq C(T,|\Omega|,s,\rho_0,M)\tau^{1/2}$.
\end{lemma}

\begin{proof}
We derive that
\begin{equation}
\begin{split}
&\;\int_{0}^{t_0-\tau} \int_{\Omega} \rho^{\tau}(x,t)\frac{\phi(x,t+\tau)-\phi(x,t)}{\tau}\,dx\,dt\\
=&\;-\int_{\tau}^{t_0-\tau} \int_{\Omega} \frac{\rho^{\tau}(x,t)-\rho^{\tau}(x,t-\tau)}{\tau}\phi(x,t)\,dx\,dt\\
&\;+\frac{1}{\tau}\int_{t_0-\tau}^{t_0}\int_{\Omega} \rho^{\tau}(x,t-\tau)\phi(x,t) \, dx\,dt -\frac{1}{\tau}\int_{0}^{\tau}\int_{\Omega} \rho^{\tau} \phi \,dx\,dt.
\end{split}
\label{eqn: summation by part}
\end{equation}

For the left hand side, by Taylor expansion,
\[
\begin{split}
&\;\left|\int_{0}^{t_0-\tau} \int_{\Omega} \rho^{\tau}(x,t)\frac{\phi(x,t+\tau)-\phi(x,t)}{\tau}\,dx\,dt - \int_{0}^{t_0} \int_{\R^d} \rho^{\tau} \partial_t\phi \,dx\,dt\right|\\
\leq &\;\int_{0}^{t_0-\tau} \rho^{\tau}(\Omega,t)\cdot\frac{\tau}{2}\|\partial_t^2 \phi\|_{L^\infty(\Omega_T)}\,dt
+\int_{t_0-\tau}^{t_0}\rho^\tau(\Omega,t)\|\partial_t \phi\|_{L^\infty(\Omega_T)}\,dt\\
\leq &\;C(T,\rho_0)\tau \|\phi\|_{C^2(\Omega_T)}.
\end{split}
\]

To handle the first term on the right hand side of \eqref{eqn: summation by part},
we use the pushforward formula to derive that
\[
\begin{split}
&\;\int_{\Omega} \frac{\rho^{\tau}(x,t)-\rho^{\tau}(x,t-\tau)}{\tau}\phi(x,t)\,dx
=
\int_{\Omega} \rho^{\tau}(x,t)\frac{\phi(x,t)-\phi\big(x+\tau\nabla p^{\tau}(x,t),t\big)}{\tau}\,dx.
\end{split}
\]
Thanks to the Taylor expansion of $\phi\big(x+\tau\nabla p^{\tau}(x,t),t\big)$ and Lemma \ref{lem: edi},
\[
\begin{split}
&\;\left|\int_{\tau}^{t_0-\tau}\int_{\Omega} \frac{\rho^{\tau}(x,t)-\rho^{\tau}(x,t-\tau)}{\tau}\phi(x,t)\,dx\,dt
+\int_{\tau}^{t_0-\tau}\int_{\Omega} \rho^{\tau}\nabla p^\tau\cdot \nabla \phi\,dx\,dt\right|\\
\leq &\; \|\phi\|_{C^2(\Omega_T)}\int_{0}^{T}\int_{\Omega} \frac{\tau}{2}\rho^{\tau}|\nabla p^\tau|^2\,dx\,dt\\
\leq &\; C(s,\rho_0)\tau \|\phi\|_{C^2(\Omega_T)}.
\end{split}
\]
Besides, by the Cauchy-Schwarz inequality and Lemma \ref{lem: edi},
\[
%\begin{split}
%&\;
\left|\left(\int_{0}^{\tau}+\int_{t_0-\tau}^{t_0}\right)\int_{\Omega} \rho^{\tau}\nabla p^\tau\cdot \nabla \phi\,dx\,dt\right|%\\
%\leq &\;\left[\left(\int_{0}^{\tau}+\int_{t_0-\tau}^{t_0}\right)\rho^\tau(\Omega,t)\,dt
%\cdot\int_{0}^T\int_{\Omega}\rho^\tau|\nabla p^\tau|^2\,dx\,dt\right]^{1/2}\|\phi\|_{C^1(\Omega_T)}\\
\leq %&\;
C(s,\rho_0)\tau^{1/2} \|\phi\|_{C^1(\Omega_T)}.
%\end{split}
\]

For the last two terms in \eqref{eqn: summation by part}, we derive that
\[
\begin{split}
&\;\left|\frac{1}{\tau}\int_{t_0-\tau}^{t_0}\int_{\Omega} \rho^{\tau}(x,t-\tau)\phi(x,t) \, dx\,dt -\int_{\Omega} \rho^{\tau}(x,t_0)\phi(x,t_0) \, dx\right|\\
\leq &\;\frac{1}{\tau}\int_{t_0-\tau}^{t_0}\rho^{\tau}(\Omega,t-\tau) \|\phi(\cdot,t)-\phi(\cdot,t_0)\|_{C(\Omega)}\,dt\\
&\;+\frac{1}{\tau}\int_{t_0-\tau}^{t_0}\|\rho^{\tau}(\cdot,t-\tau)-\rho^{\tau}(\cdot,t_0)\|_{H^{-1}(\R^d)} \|\phi(\cdot,t_0)\|_{H^1( \R^d)}\,dt.
\end{split}
\]
%where $C$ depends on the size of $\spt\phi$.
The first term is bounded by
$\tau \rho_0(\Omega)\|\phi\|_{C^1(\Omega_T)}$.
By the Cauchy-Schwarz inequality and the definition of $\rho^\tau$, the second term is bounded by
\[
\begin{split}
&\;C(|\Omega|)\|\phi\|_{C^1(\Omega_T)}\tau^{-1/2}\left(\int_{t_0-\tau}^{t_0}\|\rho^{\tau}(\cdot,t-\tau)-\rho^{\tau}(\cdot,t_0)\|_{H^{-1}(\R^d)}^2\,dt\right)^{1/2}\\
\leq &\;C(|\Omega|)\|\phi\|_{C^1(\Omega_T)}\Big(\|\rho^{\tau}(\cdot,t_0-2\tau)-\rho^{\tau}(\cdot,t_0-\tau)\|_{H^{-1}(\R^d)}^2
+\|\rho^{\tau}(t_0-\tau,\cdot)-\rho^{\tau}(t_0,\cdot)\|_{H^{-1}(\R^d)}^2\Big)^{1/2}.
\end{split}
\]
By Lemma \ref{lem:rho_time_bound}, this is further bounded by $C(|\Omega|,s,\rho_0,M)\tau^{1/2}\|\phi\|_{C^1(\Omega_T)}$.
Hence,
\[
\left|\frac{1}{\tau}\int_{t_0-\tau}^{t_0}\int_{\Omega} \rho^{\tau}(x,t-\tau)\phi(x,t) \, dx\,dt -\int_{\Omega} \rho^{\tau}(x,t_0)\phi(x,t_0) \, dx\right|\\
\leq C(|\Omega|,s,\rho_0,M)\tau^{1/2}\|\phi\|_{C^1(\Omega_T)}.
\]
Similarly, the last term in \eqref{eqn: summation by part} satisfies the same bound.
Summarizing all the above estimates, we complete the proof.
\end{proof}

We first show Theorem~\ref{thm:pde_existence1}.
\begin{proof}[Proof of Theorem~\ref{thm:pde_existence1}]
\ref{L^1 convergence positive case} is proved in Proposition~\ref{prop:strong_compactness}.

By %Theorem \ref{thm: direct comparison of rho and p} and
Lemma~\ref{lem:lower_bound}, we have uniform-in-$\tau$ lower bounds for $\rho^\tau$ and $p^\tau$, namely,
%In particular,
$\rho^\tau \geq \rho_\lambda \geq a_\lambda>0$ a.e.\;and $p^\tau\geq \alpha_\lambda$.
Here $a_\lambda$ and $\alpha_\lambda$ are defined in Lemma \ref{lem:lower_bound}.
%$\rho^\tau\geq \rho_\lambda\geq a_\lambda$ by
%Lemma~\ref{lem:lower_bound}, and a uniform upper bound by
Then Lemma \ref{lem: edi} implies $\nabla p^\tau$ is uniformly bounded in $L^2(\Omega_T)$.
By Lemma~\ref{bounds}, they are also uniformly bounded from above.
Hence, there exists $p\in L^2([0,T];H^1(\Omega))\cap L^\infty(\Omega_T)$, such that $\nabla p^\tau \rightharpoonup \nabla p$ in $L^2(\Omega_T)$ and $p^\tau \rightharpoonup p$ weak-$*$ in $L^\infty(\Omega_T)$ along a subsequence.
We may additionally assume $p(x,t)\leq M$ in $\Omega_T$ because $p^\tau\leq M$.
The convergence of $p^\tau$ together with Proposition~\ref{prop:strong_compactness} and the uniform boundedness of $\rho^\tau$ in $L^\infty(\Omega_T)$ implies that $\rho^{\tau}\nabla p^{\tau}\rightharpoonup \rho\nabla p$ in $L^2(\Omega_T)$ along a subsequence, and that $(\rho^{\tau})^{1/2}\nabla p^{\tau}\rightharpoonup \rho^{1/2}\nabla p$ in $L^2(\Omega_T)$.
Hence, $\rho |\nabla p|^2\leq \liminf_{\tau \to 0}\rho^\tau|\nabla p^\tau|^2$ a.e.\;in $\Omega_T$.

By Fubini's theorem, for a.e.\;$t_0\in [0,T]$, $\rho^\tau(\cdot,t_0)\to \rho(\cdot,t_0)$ in $L^1(\Omega)$.
So \ref{weak formulation positive case} follows from the aforementioned convergence results and Lemma~\ref{continuity}.

Lastly, to show \ref{relation between p and rho positive case}, we derive by Lemma~\ref{lem:dual_relation} and the dual relation of $(\rho^{\tau}, p^{\tau})$ that, for any non-negative $\phi\in L^\infty(\Omega_T)$,
$$
\int_0^T\int_\Omega\rho^{\tau}p^{\tau}\phi\,dx\,dt=\int_0^T\int_\Omega \big(s(\rho^{\tau}(x,t),x)+s^*(p^{\tau}(x,t),x)\big)\phi(x,t)\,dx\,dt.
$$
Since $\rho^\tau \to \rho$ in $L^1(\Omega_T)$ by Proposition \ref{prop:strong_compactness} and $p^\tau \rightharpoonup p$ weak-$*$ in $L^\infty(\Omega_T)$ along a subsequence, we have
$$
\int_0^T\int_\Omega\rho^{\tau}p^{\tau}\phi\,dx\,dt\to \int_0^T\int_\Omega\rho p\phi\,dx\,dt.
$$
On the other hand, since $\rho^\tau \to \rho$ a.e.\;in $\Omega_T$ up to a further subsequence, we use \ref{assumption: convexity} to derive that
$$
\liminf_{\tau\to 0}s(\rho^{\tau}(x,t),x)\geq s(\rho(x,t),x)\quad \mbox{a.e.\;in }\Omega_T.
$$
By \ref{assumption: effective domain} and Fatou's lemma,
$$
\liminf_{\tau\to 0}\int_0^T\int_\Omega s(\rho^{\tau}(x,t),x)\phi(x,t)\,dx\,dt \geq \int_0^T\int_\Omega s(\rho(x,t),x)\phi(x,t)\,dx\,dt.
$$
Note that Fatou's lemma is applicable here since $s(z,x)$ admits a lower bound thanks to \ref{assumption: effective domain}.
Moreover, since $s^*(\cdot,x)$ is convex and $\phi\geq 0$,
$$
\int_0^T\int_\Omega s^*(p^{\tau},x)\phi\,dx\,dt\geq \int_0^T\int_\Omega \big(s^*(p,x)+\partial_p s^*(p,x)(p^\tau-p)\big)\phi\,dx\,dt.
$$
By the convexity of $s^*(\cdot,x)$ and Lemma \ref{lem:pressure_upper_bound},
\[
0\leq \partial_p s^*(p(x,t),x)\leq \sup_{x\in\Omega}\partial_p s^*(M,x)<+\infty.
\]
Since $p^\tau \rightharpoonup p$ weak-$*$ in $L^\infty(\Omega_T)$, we obtain that
$$
\liminf_{\tau \to 0}\int_0^T\int_\Omega s^*(p^{\tau},x)\phi\,dx\,dt\geq \int_0^T\int_\Omega s^*(p,x)\phi\,dx\,dt.
$$
Combining all the convergence, we obtain that for any non-negative $\phi\in L^\infty(\Omega_T)$,
$$
\int_0^T\int_\Omega\rho p\phi\,dx\,dt\geq\int_0^T\int_\Omega \big(s(\rho,x)+s^*(p,x)\big)\phi(x,t)\,dx\,dt.
$$
Therefore, $\rho p \geq s(\rho,x)+s^*(p,x)$ a.e.\;in $\Omega_T$.
Then \ref{relation between p and rho positive case} follows from the Young's inequality and Lemma \ref{lem:dual_relation}.
\end{proof}

Next we proceed to prove Theorem~\ref{thm:pde_existence2}.
The main idea is to construct a decreasing sequence of strictly positive densities $\{\rho_k\}$ by Theorem \ref{thm:pde_existence1} to approximates $\rho$ from above, and then take the limit $k\to \infty$.

\medskip

\noindent$\circ$ {\it Construction of a monotone approximation}

\medskip

Let $\rho_{0,k}:=\max\{\rho_0,\rho_{\frac{1}{k}}\}$, %where $\{\lambda_k\}$ is a strictly decreasing  sequence of positive numbers such that $\lambda_k\to 0$ and
where $\rho_{\frac{1}{k}}$ is defined as in Lemma \ref{lem:lower_bound}.
It is straightfrward to check that $\rho_{0,k}$ (with $k$ sufficiently large) all satisfy \eqref{good_data_1} with a uniform $M$, and $E(\rho_{0,k})$ are uniformly bounded.
Besides, $\rho_{0,k}$ is non-increasing in $k$ with $\rho_{0,k}\to \rho_0$ a.e.\;in $\Omega$.

Let $\rho^{\tau}_k$ and $p^{\tau}_k$ denote the discrete density and pressure pair corresponding to the initial density $\rho_{0,k}$. Then Theorem~\ref{thm:pde_existence1} applies to $(\rho^{\tau}_k, p^{\tau}_k)$. Let $(\rho_k,p_k)$ be a corresponding continuum limit obtained along a subsequence of $\tau\to 0$; here for all $k$ we assume the convergence holds along a common subsequence of $\tau$, which can be extracted by a diagonalization argument.

Since $\rho_{k}^\tau(\cdot,t)\in X$ satisfies \eqref{good_data_1} (and thus \eqref{eq:good_data0}) for all $k$ and $t\geq 0$, Theorem \ref{thm: direct comparison of rho and p} yields that $\rho^{\tau}_k$ and $p^{\tau}_k$ are monotone decreasing with respect to $k$, and Lemma \ref{bounds} implies that $\rho_k^\tau$ are uniformly bounded in $L^\infty(\Omega_T)$.
Hence, the same properties hold for $\rho_k$ and $p_k$ thanks to Theorem \ref{thm:pde_existence1}.
Furthermore, from Lemma \ref{lem: edi} we have
\begin{equation*}%\label{unif_bd}
\sup_k \int_0^\infty\int_{\Omega}\rho^{\tau}_k(x,t)|\nabla p^{\tau}_k(x,t)|^2\, dx\,dt <\infty.
\end{equation*}
By Theorem \ref{thm:pde_existence1}\ref{convergence of p tau to p with positive density},
\begin{equation}\label{unif_bd}
\sup_k \int_0^\infty\int_{\Omega}\rho_k(x,t)|\nabla p_k(x,t)|^2\, dx\,dt <\infty.
\end{equation}

By virtue of Proposition \ref{prop:strong_compactness}, let $\rho\in L^\infty(\Omega_T)$ be the $L^1(\Omega_T)$-limit of $\rho^\tau$ along a further subsequence.
Since Theorem \ref{thm:l1_contraction} gives that, for all $t\geq 0$,
$$
\norm{(\rho^{\tau}-\rho^{\tau}_k)(\cdot,t)}_{L^1(\Omega)}\leq \frac{1}{k},
$$
using the strong $L^1$-convergence of $\rho^{\tau}$ and $\rho^{\tau}_k$ as $\tau \to 0$, we conclude that
$$
\norm{\rho-\rho_k}_{L^1(\Omega_T)}\leq\frac{T}{k},
$$
which implies $\rho_k\to \rho$ a.e.\;in $\Omega_T$.
Lastly, we let $p$ in $\Omega_T$ be the pointwise limit of the decreasing sequence $p_k$. Since $p_k$ are uniformly bounded from above thanks to Theorem \ref{thm:pde_existence1}, $p_+\in L^\infty(\Omega_T)$.

\medskip

When proving Theorem \ref{thm:pde_existence2}, we shall denote $m(x,t)$ to be a weak limit of $\rho_k\nabla p_k$ as $k\to \infty$.
In order to justify the characterization of $m$ in Theorem \ref{thm:pde_existence2}\ref{relation between m and grad p}, we need to further study the convergence of $\rho_k\nabla p_k$.

\begin{lemma}\label{convergence:further}
Let $\rho_k, p_k$ and $\rho$ be as above.  Then $\rho_k \nabla p_k$  is uniformly bounded in $L^2(\Omega_T)$ and converges weakly along a subsequence to a vector field $m\in L^2(\Omega)$ as $k\to\infty$.
Moreover, $m$ is absolutely continuous with respect to $\rho$, and satisfies \eqref{int_parts}.
\end{lemma}
\begin{proof}

Let $m_k=\rho_k\nabla p_k$. Since $\rho_k$ is uniformly bounded in $L^{\infty}(\Omega_T)$, it follows from \eqref{unif_bd} that $m_k$ is uniformly bounded in $L^2(\Omega_T)$.
So there exists $m\in L^2(\Omega_T)$ and a subsequence of $\{m_k\}$, which is still denoted by $m_k$ with abuse of notations, such that $m_{k}\rightharpoonup m$ weakly in $L^2(\Omega_T)$.
So $|m|^2\leq \liminf_{k\to\infty}|m_k|^2$ a.e.\;in $\Omega_T$.
Then by the fact $\rho_k\to \rho$ a.e.\;in $\Omega_T$ and the Fatou's lemma, 
\[
\int_{\Omega_T} \frac{|m|^2}{\rho}\leq \liminf_{k\to\infty} \int_{\Omega_T} \frac{|m_k|^2}{\rho_k}<\infty.
\]
Hence, by the Cauchy-Schwarz inequality, $m$ is absolutely continuous with respect to $\rho$ in $\Omega_T$.

Let $f \in L^2([0,T];H^1(\Omega))$ be a vector field with zero normal component on $\partial\Omega\times [0,T]$, such that $f\rho^{-1}\in L^{\infty}(\Omega_T)$ and $\rho^{-1}\nabla\cdot f\in L^1(\Omega_T)$.
Indeed, such $f$ exists.
For instance, let $\eta_{\delta}:\RR\to\RR$ be a smooth function such that
\[
\eta_{\delta}(q)=
\begin{cases}
1 &\textrm{if}\quad \essinf_{x\in\Omega} \partial_p s^*(q,x)\geq \delta,\\
0 &\textrm{if}\quad \essinf_{x\in\Omega} \partial_p s^*(q,x)\leq \delta/2.\\
\end{cases}
\]
We claim that $f(x,t):=\eta_\delta (p)F(x,t)$ satisfies the desired properties, where $F(x,t)$ is an arbitrary smooth vector field in $\Omega_T$ with zero normal component on $\partial\Omega\times [0,T]$.
By definition, $\eta_\delta(p_k)\to \eta_\delta(p)$ pointwise in $\Omega_T$ and $\eta_\delta(p)\in L^\infty(\Omega_T)$.
Next, by the relation $\rho_k(x,t)=\partial_p s^*(p_k(x,t),x)$ a.e., we see that  $\{ \rho_k(x,t)\leq \frac{\delta}{2}\} \subset \{\eta_{\delta}(p_k)= \eta'_{\delta}(p_k) \equiv 0\}$ up to a measure zero set in $\Omega_T$, so $|\rho_k^{-1}\eta_\delta(p_k)|\leq 2\delta^{-1}$ a.e..
This implies $\rho^{-1}\eta_\delta(p)\leq 2\delta^{-1}$ a.e.\;in $\Omega_T$ and thus $f\rho^{-1}\in L^\infty(\Omega_T)$.
Moreover,
\[
\int_{\Omega_T} |\nabla \big(\eta_{\delta}(p_k)\big)|^2\leq \frac{2}{\delta}\norm{\eta_{\delta}'}_{L^{\infty}(\RR)}\int_{\Omega_T} \rho_k|\nabla p_k|^2,
\]
which is uniformly bounded.
Hence, combined with $\eta_\delta(p_k)\to \eta_\delta(p)$, we know that $\nabla (\eta_\delta(p_k))\rightharpoonup \nabla (\eta_\delta(p))$ in $L^2(\Omega_T)$.
It is then clear $f\in L^2([0,T];H^1(\Omega))$.
Since
\[
\int_{\Omega_T} \rho_k^{-2} |\nabla (\eta_{\delta}(p_k) )|^2\leq \left(\frac{2}{\delta}\right)^{3}\norm{\eta_{\delta}'}_{L^{\infty}(\RR)}\int_{\Omega_T} \rho_k|\nabla p_k|^2,
\]
we apply Fatou's lemma to find that $\rho^{-1}\nabla(\eta_\delta(p))\in L^2(\Omega_T)\subset L^1(\Omega_T)$.
Therefore, we can conclude that $\rho^{-1}\nabla \cdot f\in L^1(\Omega_T)$.

We can then estimate %\textcolor{red}{(in order to have this step hold, the assumption $f\rho^{-1}\in L^\infty$ can be replaced by any $L^r$ with $r>2$)}
\[
\begin{split}
\lim_{k\to\infty}\Big|\int_{\Omega_T}\Big( \frac{\rho_k}{\rho} -1\Big) \nabla p_k\cdot f\Big|\leq&\;
\lim_{k\to\infty} \norm{f\rho^{-1}}_{L^{\infty}(\Omega_T)}\norm{(\rho_k-\rho)^{1/2}\nabla p_k}_{L^2(\Omega_T)}\norm{\rho_k-\rho}_{L^{1}(\Omega_T)}^{1/2}\\
\leq&\;\norm{f\rho^{-1}}_{L^{\infty}(\Omega_T)}\lim_{k\to\infty} \norm{\rho_k^{1/2}\nabla p_k}_{L^{2}(\Omega_T)}\norm{\rho_k-\rho}_{L^{1}(\Omega_T)}^{1/2}=0.
\end{split}
\]
Thus,
\[
\int_{\Omega_T} \frac{m}{\rho}\cdot f =\lim_{k\to\infty}\int_{\Omega_T} \frac{\rho_{k}\nabla p_{k}}{\rho}\cdot f=\lim_{k\to\infty}\int_{\Omega_T} \nabla p_{k}\cdot f.
\]
Recall that by Theorem \ref{thm:pde_existence1}, $p_k\in L^2([0,T];H^1(\Omega))$.
It is then legitimate to integrate by parts to obtain
\[
\int_{\Omega_T} \frac{m}{\rho}\cdot f =-\lim_{k\to\infty}\int_{\Omega_T} p_{k}\nabla\cdot f=-\lim_{k\to\infty}\int_{\Omega_T} \rho p_{k}\cdot \rho^{-1}\nabla\cdot f.
\]
By Lemma \ref{bounds}, $\rho p\in L^{\infty}(\Omega_T)$ and $\rho_k p_k\in L^{\infty}(\Omega_T)$ where the bound only depends on $s$ and $M$ in \eqref{good_data_1}.  Since $\rho\leq \rho_k$, it also follows that $\rho p_k$ are uniformly bounded in $L^{\infty}(\Omega_T)$ and $\rho p_k \to \rho p$ pointwise.
By the dominated convergence theorem,
\begin{equation}\label{eq:distribution_result}
\int_{\Omega_T} \frac{m}{\rho}\cdot f=-\int_{\Omega_T} p \nabla \cdot f.
\end{equation}
Note that the left hand side is well defined as long as $\rho^{-1}f\in L^2(\Omega_T)$, while the right hand side is well defined as long as $\rho^{-1}\nabla \cdot f\in L^1(\Omega_T)$. Thus, by a limiting argument, we see that \eqref{eq:distribution_result} holds for all $f\in \mathcal{T}$.
\end{proof}

\begin{proof}[Proof of Theorem~\ref{thm:pde_existence2}]
From Theorem~\ref{thm:pde_existence1} and Lemma~\ref{convergence:further}, we can send $k\to\infty$ in \eqref{weak} applied to $(\rho_k, p_k)$ to obtain \eqref{weak_eq},
where $m \in L^2(\Omega_T)$ satisfies \eqref{int_parts}.

It remains to show \ref{relation between p and rho} and
\begin{equation}\label{claim000}
m = \nabla s^*(p(x,t),x) - \partial_x s^*(p(x,t),x).
\end{equation}
Indeed,
%Lemma~\ref{lem:dual_relation} and
applying Theorem~\ref{thm:pde_existence1}\ref{relation between p and rho positive case} to $(\rho_k,p_k)$ yields
\begin{equation}\label{dual_k_2}
 (\rho_k\nabla p_k)(x,t) = \nabla s^*(p_k(x,t),x)- \partial_x s^*(p_k(x,t),x),
\end{equation}
which is uniformly bounded in $L^2(\Omega_T)$.
We claim that $s^*(p_k(x,t),x)$ is uniformly bounded in $L^2([0,T];H^1(\Omega))$.
Indeed, on one hand, the $L^2$-bound of $\nabla s^*(p_k(x,t),x)$ is from \eqref{dual_k_2} and the assumption \ref{assumption: regularity of s*}.
On the other hand, $s^*(p_k(x,t),x)$ is uniformly bounded from above because of the fact $p_k\leq M$ (Theorem \ref{thm:pde_existence1}) and Lemma \ref{lem:pressure_upper_bound}.
By Lemma \ref{lem:s_star_increasing}, $s^*(p_k(x,t),x)$ is also non-negtive. 
This justifies the claim.
Therefore, up to a subsequence, $s^*(p_k(x,t),x)$ converges weakly to some $w$ in $L^2([0,T];H^1(\Omega))$.  Since $p_k\to p$ pointwise, it follows that $w = s^*(p(x,t),x)$, and $\nabla s^*(p_k(x,t),x)\rightharpoonup\nabla s^*(p(x),x)$ in $L^2(\Omega_T)$. Lastly, due to the pointwise convergence of $p_k\to p$ and \ref{assumption: regularity of s*}, we can conclude that $\partial_x s^*(p_k(x,t),x)$ weakly converges to $\partial_x s^*(p(x,t),x)$ in $L^2(\Omega_T)$.  Now sending $k\to\infty$ in \eqref{dual_k_2} yields \eqref{claim000}.

Lastly, \ref{relation between p and rho}  follows from  sending $k\to\infty$ in the dual realtion
$$
(\rho_kp_k)(x,t) = s(\rho_k(x,t),x) + s^*(p_k(x,t),x)\quad a.e.\;\mbox{ in }\Omega_T,
$$
using a.e.\;convergence in each term and Lemma~\ref{lem:dual_relation}.
\end{proof}

\section{Uniqueness of Weak Solutions}
\label{sec: uniqueness}
In this section we discuss uniqueness of the weak solutions constructed in the previous section. Our proof largely follows that of the $x$-independent case (see \cite{Carrillo,vazquez}).

\begin{definition}\label{weak_sol}
$(\rho, p)$ is a weak solution of $(P)$ if for any $T>0$, $\rho\in L^{\infty}\big(\Omega_T\big)$, $p$ is measurable with $p_+ \in L^{\infty}(\Omega_T)$, %$m(x,t) := \nabla [s^*(p(x,t),x)] - \partial_x s^*(p(x,t),x)\in L^2(\Omega_T)$,
and they satisfy \ref{relation between p and rho} and \ref{weighted_pressure} in Theorem~\ref{thm:pde_existence2}.
\end{definition}

In Theorem~\ref{thm:pde_existence2}, we have shown the existence of weak solutions under the assumption that $\partial_x s^*(\cdot, x)$ is continuous and satisfies \ref{assumption: regularity of s*}.
To discuss the uniqueness of the weak solution, we need a stronger assumption on the continuity of  $\partial_x s^*$: there exists a constant $C>0$ such that for all $x\in \Omega$,
\begin{equation}\label{assumption_s}
|\partial_x s^*(p_1, x) - \partial_x s^*(p_2, x) |^2 \leq  C | \partial_p s^*(p_1, x) - \partial_p s^*(p_2,x)| |s^*(p_1,x) - s^*(p_2,x)|.
\end{equation}

\begin{remark}
The notion of weak solution in Definition \ref{weak_sol} is the same as the one given in \cite{Carrillo} for the case $s^*=s^*(p)$, where \eqref{assumption_s} is automatically satisfied.

\end{remark}

\begin{theorem}[Uniqueness]\label{thm:uniqueness}
Under the assumption \eqref{assumption_s} in addition to \ref{assumption: convexity}-\ref{assumption: inf sup ratio of partial s*}, there exists at most one weak solution of $(P)$, in the sense that if $(\rho_1,p_1)$ and $(\rho_2, p_2)$ are two pairs of weak solutions for $(P)$, then  we have that  $\rho_1=\rho_2$ a.e.\;in $\Omega_T$ and $m_1=m_2$ a.e.\;in $\Omega_T$. Here $m_i$ are defined in Theorem~\ref{thm:pde_existence2}\ref{weighted_pressure} with $p_i$.
\end{theorem}
\begin{remark}
\begin{enumerate}%[label=(\arabic*)]
\item When $s^*$ has the special form $s^*(p,x)=f(x)w(p)$ with positive $f$, then the assumption (\ref{assumption_s}) is not needed.  Indeed, if one changes the choice of test function in the proof to $\phi(x)=w(p_1(x))-w(p_2(x))$, the ``bad'' term  $\partial_x s^*(p_1, x) - \partial_x s^*(p_2, x)$ does not appear and the remaining terms have the correct sign.
\item Note that uniqueness of weak solutions implies that the convergence results in Theorems \ref{thm:pde_existence1} and \ref{thm:pde_existence2} must hold along the full sequence $\rho^{\tau}$.  As a result, one can use the discrete scheme to conclude that solutions to the continuum PDE satisfy the $L^1$-contraction principle.
\end{enumerate}
\end{remark}

\begin{proof}
Let $(\rho_i,p_i)$, $i=1,2$ as given above. Then by a density argument, for any $\phi \in H^1(\Omega)$ and for a.e.\;$t>0$ we have the weak formulation
$$
\int_0^{t} \int_{\Omega} m_i \nabla \phi\, dx\,ds = \int_{\Omega} (\rho_0(x) - \rho_i(x,t) )\phi(x) \,dx,
$$
where $m_i (x,t) := \nabla s^*(p_i(x,t), x) - \partial_x s^*(p_i(x,t),x)$. Taking the difference of the respective weak formulations of $i = 1$ and $i = 2$, we have
\begin{equation}\label{above}
\int_{\Omega} (\rho_1(x,t) - \rho_2(x,t))\phi(x)\,  dx=  - \int_0^{t} \int_{\Omega} (m_1 - m_2)\cdot\nabla \phi \,dx\, ds.
\end{equation}

Take $\phi(x):= s^*(p_1(x,t), x) - s^*(p_2(x,t),x)$, which is in $H^1(\Omega)$ for a.e.\;$t>0$ (see the proof of Theorem \ref{thm:pde_existence2}).
Then
\begin{equation*}
\nabla\phi = (m_1- m_2)(x,t) + \partial_x s^*(p_1(x,t),x) - \partial_x s^*(p_2(x,t),x).
\end{equation*}
Define $h(x,t):= \int_0^{t}(m_1-m_2)(x,s)\,ds$, and we may rewrite \eqref{above} as
$$
\int_{\Omega} (\rho_1-\rho_2)(x,t) \phi(x)\, dx = - \int_{\Omega} h(x,t) \big(\partial_t h(x,t) + \partial_x s^*(p_1(x,t),x) - \partial_x s^*(p_2(x,t),x)\big) \,dx.
$$
By Cauchy-Schwarz,
$$
\int_{\Omega} (\rho_1-\rho_2)(x,t) \phi(x) \,dx + \int_{\Omega} h(x,t) \partial_t h(x,t)  dx \leq \frac{1}{2\delta} \int_{\Omega} h^2(x,t) dx + \frac{\delta}{2} \int_{\Omega} g^2(x,t) \,dx.
$$
where $g(x,t):= \partial_x s^*(p_1(x,t),x) - \partial_x s^*(p_2(x,t),x)$ and $\delta >0$ is to be determined. Integrating in time, we obtain
$$
\int_0^T\int_{\Omega} (\rho_1-\rho_2)(x,t) \phi(x)\, dx + \frac{1}{2}\int_{\Omega} h^2(x,T) \, dx \leq \frac{1}{2\delta} \int_0^T\int_{\Omega} h^2(x,t) \,dx\,dt + \frac{\delta}{2} \int_0^T \int_{\Omega}  g^2(x,t)\, dx\, dt.
$$

From the dual relation $\rho_i (x,t) = \partial_p s^*(p_i(x,t),x)$ and the fact that $s^*$ is monotone increasing and convex in $p$ (Lemma~\ref{lem:s_star_increasing}), we have
$$
(\rho_1-\rho_2)(x,t) \phi(x) = |\rho_1(x,t)-\rho_2(x,t)||s^*(p_1(x,t),x) - s^*(p_2(x,t),x)|.
$$
Also, \eqref{assumption_s} and the dual relation imply that for some $C_*>0$,
$$
g^2(x,t) \leq C_*| \rho_1(x,t) - \rho_2(x,t)| |s^*(p_1,x) - s^*(p_2,x)|.
$$
Hence by choosing $\delta = C_*^{-1}$, we have
$$
\int_0^T\int_{\Omega} |(\rho_1-\rho_2)(x,t)| |s^*(p_1,x) - s^*(p_2,x)| \,dx\,dt + \int_{\Omega} h^2(x,T)\,dx \leq C_* \int_0^T\int_{\Omega} h^2(x,t)\, dx\,dt.
$$
By Gronwall's inequality, for all $T>0$,
\[
\int_0^T\int_{\Omega} h^2(x,t)\, dx\,dt = 0.
\]
This implies $h(x,t) = 0$ a.e.\;in $\Omega_T$ and hence $m_1=m_2$ almost everywhere.  Now it follows from equation (\ref{above}) that $\rho_1=\rho_2$ almost everywhere.
\end{proof}

\bibliographystyle{amsalpha}
%\bibliography{l1_contraction_refs}

\providecommand{\bysame}{\leavevmode\hbox to3em{\hrulefill}\thinspace}
\providecommand{\MR}{\relax\ifhmode\unskip\space\fi MR }
\providecommand{\MRhref}[2]{%
  \href{http://www.ams.org/mathscinet-getitem?mr=#1}{#2}
}
\providecommand{\href}[2]{#2}

\end{document}